\documentclass[12pt]{article}
\usepackage{amssymb}
\usepackage{enumerate}
\usepackage{amsfonts}
\usepackage{amsmath}
\usepackage{amsthm}
\usepackage{psfrag} 
\usepackage{epsfig}
\usepackage{subfigure}
\usepackage{graphicx}
\usepackage{color}
\usepackage{amssymb}
\usepackage{epstopdf}
\usepackage{amsbsy}
\usepackage{latexsym}
\usepackage{moreverb}                      
\usepackage{textcomp}
\usepackage{multicol}
\usepackage{algorithm}
\usepackage{algorithmic}
\usepackage{enumitem}

\newtheorem{remark}{Remark}
\newtheorem{definition}{Definition}
\newtheorem{theorem}{Theorem}
\newtheorem{corollary}{Corollary}

\newcommand{\ubar}[1]{\text{\b{$#1$}}}

\usepackage[top=2.6cm,bottom=2.6cm,left=2cm,right=2cm]{geometry}

\usepackage[T1]{fontenc}
\usepackage[utf8]{inputenc}

\usepackage[affil-it]{authblk}
\usepackage{lmodern}


\usepackage{sectsty}
\sectionfont{\fontsize{13}{15}\selectfont}
\subsectionfont{\fontsize{12}{15}\selectfont}
\begin{document}

\title{Fuzzy-Stochastic Partial Differential Equations}

\author[1]{Mohammad Motamed\thanks{motamed@math.unm.edu}}

\affil[1]{Department of Mathematics and Statistics, The University of New Mexico, Albuquerque, USA}
\maketitle

\begin{abstract}
We introduce and study a new class of partial differential equations (PDEs) with hybrid fuzzy-stochastic parameters, coined {\it fuzzy-stochastic PDEs}. Compared to purely stochastic PDEs or purely fuzzy PDEs, fuzzy-stochastic PDEs offer powerful models for accurate representation and propagation of hybrid aleatoric-epistemic uncertainties inevitable in many real-world problems. We will use the level-set representation of fuzzy functions and define the solution to fuzzy-stochastic PDE problems through a corresponding parametric problem, and further present theoretical results on the well-posedness and regularity of such problems. We also propose a numerical strategy for computing output fuzzy-stochastic quantities, such as fuzzy failure probabilities and fuzzy probability distributions. 
We present two numerical examples to compute various fuzzy-stochastic quantities and to demonstrate the applicability of fuzzy-stochastic PDEs to complex engineering problems. 
\end{abstract}

{\bf keywords} 
fuzzy-stochastic partial differential equation, uncertainty quantification, aleatoric uncertainty, epistemic uncertainty, fuzzy probability distribution, fuzzy-stochastic computation


\section{Introduction}
\label{sec:intro}

Two types of uncertainty that are often addressed in the field of uncertainty quantification (UQ) include: 1) {\it aleatoric} uncertainty that arises from inherent randomness or variability in a system, and 2) {\it epistemic} uncertainty that arises from insufficient and/or inaccurate information about a system ; see e.g. \cite{UQ_book:15}. 
Many real-world problems indeed exhibit both types of aleatoric and epistemic uncertainties. A typical example is the dynamic response of composite materials, such as carbon fiber polymers, where uncertainty in material properties and damage parameters has contributions from both types \cite{BM:16}. 
On the one hand, variations in material properties (such as the modulus of elasticity) and the spatial distribution of fiber constituents are of random nature. 
On the other hand, material constants may be either difficult or impossible to measure. Moreover, materials may come from different manufacturers, and there may be large variations in their quality leading to large variations in the experimental measurements. When experiments cannot be performed, material constants must be obtained from the literature, such as handbooks and standards. There may be large disagreement in the literature for the values of these quantities. For example, see \cite{Babuska_Silva:11} that studies large variations in the thermal conductivity of stainless steel AISI 304, based on the data given in various sources \cite{ASM:02,TPRC:70}. Such experimental and literature-based variations, which may be larger than the intrinsic random experimental noise, will introduce epistemic uncertainty. 



A major difficulty that arises in modeling uncertainty in real-world problems is that there is often no clear-cut distinction between aleatoric and epistemic uncertainties. There may be a random quantity whose parameters are partially known, or there may be an epistemically uncertain quantity for which some values are more likely to occur than others. Consequently, it may not be possible to identify different types of uncertainty and represent each type by a different model. 
Such mixture of aleatoric and epistemic uncertainties, referred to as {\it hybrid} uncertainty, may be better represented by a hybrid model obtained by the synthesis of two models, rather than by simply adding them.

There are several approaches to describe hybrid uncertainty in UQ problems. One approach within the framework of {\it imprecise probability} is to synthesize interval analysis \cite{Interval:01,Interval:09} and probability theory and build up interval-valued probability distributions \cite{Weichselberger:2000}. For instance, to an uncertain quantity we may assign a probability distribution with parameters that are represented by closed intervals describing the incomplete knowledge of the parameters. 
This approach constructs a probability-box (or a {\it p-box}) consisting of a family of cumulative distribution functions (CDFs). The left and right envelopes of the family will form a {\it box} and bound the ``unknown'' distribution of the uncertain quantity from above and below. We also refer to a recent approach, known as optimal UQ \cite{OUQ:2013}, where the optimal distribution among the family is targeted and obtained as the candidate for the ``unknown'' distribution. Other related approaches in the framework of imprecise probability include coherent lower and upper previsions \cite{Walley:91}, second-order hierarchical probabilities through Bayesian hierarchical modeling \cite{Gelman_etal:04}, and the Dempster-Shafer theory of belief functions \cite{Shafer:76}. Another approach that goes beyond the framework of probability is to synthesize probability theory and fuzzy set theory \cite{Zadeh:65} and form a hybrid fuzzy probabilistic framework for UQ; see e.g. \cite{Zadeh:84,Moller_Beer:2004,Buckley:06,Couso_etal:14}. In a fuzzy set elements can partially be in the set. Each element is hence assigned a grade (or degree) of membership. This notion can be exploited to represent an epistemically uncertain parameter by a set of nested intervals with different membership degrees. This hybrid approach builds a fuzzy probability distribution that may be regarded as a nested set of p-boxes at different levels of possibility (corresponding to different membership degrees). We refer to \cite{Walley:2000,Dubois_Prade:2010} for a general discussion of the subject.


In this paper we consider the fuzzy probabilistic approach to UQ and introduce and study PDEs with fuzzy-stochastic parameters, coined {\it fuzzy-stochastic PDEs}. Such hybrid PDEs will serve as the underlying mathematical models for physical systems subject to hybrid uncertainties. Compared to purely stochastic PDEs (see e.g. \cite{BNT:07,Motamed_etal:13,Motamed_etal:15}) and purely fuzzy PDEs (see e.g. \cite{FPDE:99,Corveleyn_Rosseel_Vanderwalle:13,FPDE:13}), fuzzy-stochastic PDEs offer powerful tools and models for accurate description and propagation of hybrid uncertainties. We use the level-set representation of fuzzy functions and define the solution to a fuzzy-stochastic PDE problem through a corresponding parametric problem. We further present theoretical results on the existence, uniqueness, and regularity of the solution. 
We develop a numerical approach for computing fuzzy-stochastic quantities of interest (QoIs), such as fuzzy failure probabilities and fuzzy probability distributions, related to the solution of fuzzy-stochastic PDEs. Considering the notion of full interaction between fuzzy variables and incorporating it in the proposed numerical approach, we avoid overestimating the lower and upper bounds of output intervals, a problem known as dependency phenomenon \cite{Interval:09} that may occur in interval arithmetic and fuzzy computations. 
We present two numerical examples. In the first example, we will compute and visualize various types of fuzzy-stochastic QoIs being functionals of the PDE solution. In the second example, we will demonstrate the importance and applicability of fuzzy-stochastic PDEs for an engineering problem in materials science: the response of fiber-reinforced polymers to external forces.

The main contributions of this work include: 1) introducing fuzzy-stochastic PDEs and defining their solution; 2) presenting well-posedness and regularity analysis of such hybrid PDEs; and 3) developing a numerical algorithm for computing fuzzy-stochastic QoIs, taking into account the full interaction between fuzzy variables. 

The rest of the paper is organized as follows. Section \ref{sec:preliminaries} provides the mathematical and computational foundations of fuzzy and fuzzy-stochastic quantities necessary for and relevant to the focus of this work. In Section \ref{sec:SFPDE} we introduce fuzzy-stochastic PDEs, define their solution, and discuss the existence, uniqueness and regularity of the solution. In Section \ref{sec:numerics} we present two numerical examples, followed by a general discussion on the computational cost of fuzzy-stochastic PDE problems. Finally, we summarize conclusions and outline future works in Section \ref{sec:CON}.

\section{Fuzzy and Fuzzy-Stochastic Quantities}
\label{sec:preliminaries}

%
This section provides the mathematical and computational foundations of fuzzy and fuzzy-stochastic quantities. 
Only the concepts relevant to the focus of this work are discussed here. 
We refer to \cite{Dubois_Prade:1980,Klir:06,Moller_Beer:2004,Buckley:06,Couso_etal:14} for a more general description of fuzzy set theory and fuzzy randomness from an engineering point of view.

\vspace{-.1cm}
\subsection{Fuzzy variables}
\label{sec:F1}

Fuzzy sets \cite{Zadeh:65}, or {\it fuzzy variables}, generalize the classical concept of a set, which in this context is referred to as a crisp set. In a crisp set, the membership of an element is given by the characteristic function, taking values either 0 (not a member) or 1 (a member). In a fuzzy set, elements can {\it partially} be in the set. Each element is given a membership degree, ranging from 0 to 1; see Figure \ref{fuzzy_crisp}. 
\begin{figure}[!h]
\vskip -.4cm
  \begin{center}
        \includegraphics[width=0.4\linewidth]{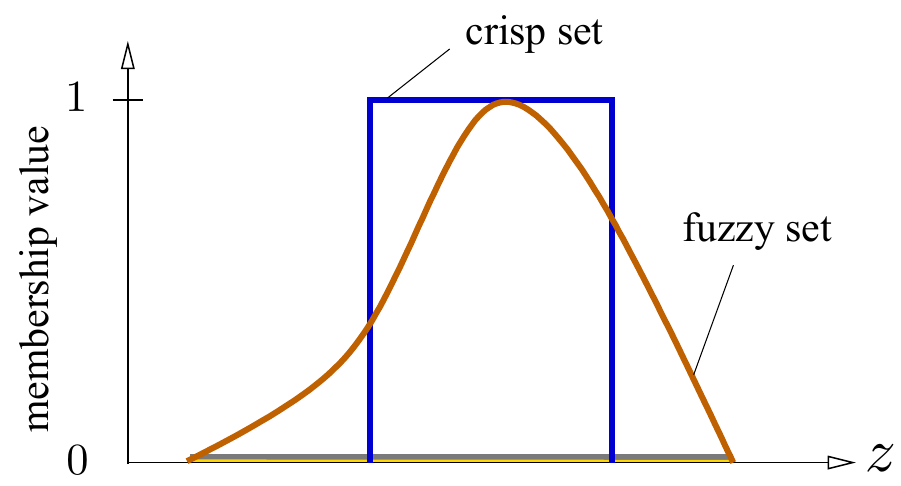}
\vskip -0.15cm        
\caption{The notion of membership in crisp sets and fuzzy sets.}
\label{fuzzy_crisp}
\vskip -.3cm
  \end{center}
\end{figure}

\begin{definition}\label{def_fuzzy_var}
A fuzzy variable is defined by a set of pairs 
$$
\tilde{z} = \{ (z, \mu_{\tilde{z}}(z)), \, \, z \in Z \subset {\mathbb R}, \, \, \mu_{\tilde{z}}: Z \rightarrow [0,1] \},
$$
where $Z$, referred to as the universe, is a non-empty subset of the real line ${\mathbb R}$, and $\mu_{\tilde{z}}$ is a membership function defined on $Z$ with range $[0,1]$. 
The set of all fuzzy variables defined on $Z$ is denoted by ${\mathcal F}(Z)$.
\end{definition}

It is to be noted that in general the range of the membership function may be a subset of nonnegative real numbers whose supremum is finite. However, it is always possible to normalize the range to $[0,1]$. Such fuzzy variables, considered here, are sometimes referred to as normalized fuzzy variables. 
Throughout the present paper, a (normalized) fuzzy variable is denoted by the superimposition of a tilde over a letter.

An important notion in fuzzy set theory is the notion of {\it $\alpha$-cuts} (see e.g. \cite{Dubois_Prade:1980,Klir:06}), which allows one to decompose fuzzy computations into several interval computations. 

\begin{definition}\label{def_alpha_cut} 
The {\it $\alpha$-cuts} of the membership function $\mu_{\tilde{z}}$ of a fuzzy variable $\tilde{z} \in {\mathcal F}(Z)$ are the family $\{ S_{\alpha}^{\tilde{z}} \subset Z,\, \alpha \in [0,1] \}$ 
of its $\alpha$-level sets:
$$
S_0^{\tilde{z}} = {\text{closure}}\{ z \in Z \, | \, \mu_{\tilde{z}}(z) >0 \}, \ \ \ \text{and} \ \ \ 
S_{\alpha}^{\tilde{z}} =  \{ z \in  Z \, | \, \mu_{\tilde{z}}(z) \ge \alpha \}, \ \ \ \forall \alpha \in (0,1].
$$
\end{definition}

In the present work we will need the following assumptions:
\begin{itemize}
\item[(A1)] \ \ The universe $Z \subset {\mathbb R}$ is a bounded interval.

\item[(A2)] \ \ The membership function is upper semicontinuous, i.e.
$$
\limsup_{z \rightarrow z_0}  \mu_{\tilde{z}}(z) \le \mu_{\tilde{z}}(z_0), \qquad \forall z_0 \in Z.
$$

\item[(A3)] \ \ The membership function is quasi-concave, i.e.
$$
\mu_{\tilde{z}} (\lambda \, z_1 + ( 1- \lambda ) \, z_2) \ge \min ( \mu_{\tilde{z}}(z_1), \mu_{\tilde{z}}(z_2)),  \qquad  \forall \, z_1, z_2 \in Z, \qquad \forall \, \lambda \in [0,1].
$$ 
\end{itemize}
We denote by ${\mathcal F}_c(Z)$ the set of all fuzzy variables $\tilde{z} \in {\mathcal F}(Z)$ satisfying assumptions (A1)-(A3). 
Note that the boundedness of $Z$ in (A1) implies that the $\alpha$-cuts are bounded sets. This is a natural assumption for the physical quantities to be represented by fuzzy variables, as such quantities are usually bounded. 
Moreover, assumptions (A2) and (A3) imply that the $\alpha$-cuts are closed and convex sets, respectively. 
Hence, for a fuzzy variable $\tilde{z} \in {\mathcal F}_c(Z)$, the $\alpha$-cuts $S_{\alpha}^{\tilde{z}}$ will be bounded, closed intervals with the inclusion property 
$S_{\alpha_2}^{\tilde{z}} \subset S_{\alpha_1}^{\tilde{z}}$ for $0 \le \alpha_1 \le \alpha_2 \le 1$;  
see Figure \ref{alpha_cut_fuzzy}.

\begin{figure}[!h]
  \begin{center}
        \includegraphics[width=0.36\linewidth]{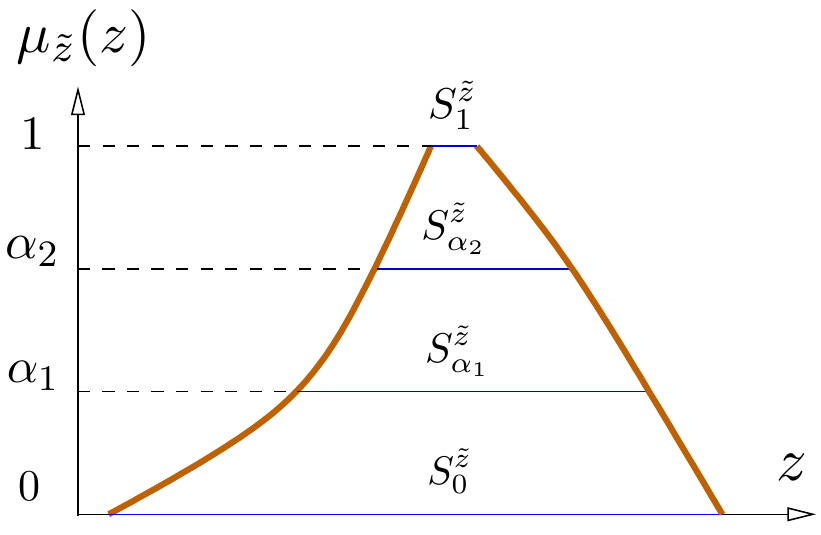}
\caption{The $\alpha$-cuts of $\tilde{z} \in {\mathcal F}(Z)$ satisfying (A1)-(A3) are closed, bounded, nested intervals.}
\label{alpha_cut_fuzzy}
  \end{center}
\end{figure}



We also define the following relational operators for fuzzy variables (see e.g. \cite{Dubois_Prade:1980}), needed for the boundedness and positivity assumption \eqref{Assum2_a} needed in Section \ref{sec:SFPDE}.
\begin{definition}\label{greater_smaller_fuzzy} 
A fuzzy variable $\tilde{z} \in {\mathcal F}(Z)$ is greater than or equal to a real number $a$ if $\mu_{\tilde{z}}(z) = 0$, $\forall z \in Z$ such that $z<a$. This is denoted by $\tilde{z} \ge a$. Similarly, a fuzzy variable $\tilde{z} \in {\mathcal F}(Z)$ is smaller than or equal to a real number $a$ if $\mu_{\tilde{z}}(z) = 0$, $\forall z \in Z$ such that $z>a$. This is denoted by $\tilde{z} \le a$. 
\end{definition}

We note that the above relations can also be expressed in terms of the zero-cut. A fuzzy variable is greater (respectively smaller) than a real number if the real number is smaller (respectively greater) than all points on the zero-cut of the fuzzy variable.  



\medskip
\noindent
{\bf Fuzzy vectors.} 
A fuzzy vector can be considered as the $n$-dimensional generalization of a fuzzy variable, with $n \ge 2$.
\begin{definition}\label{def_fuzzy_vec}
An $n$-dimensional fuzzy vector, witten as $\tilde{\bf z} = (\tilde{z}_1, \dotsc, \tilde{z}_n)$, is a collection of $n$ fuzzy variables 
$\tilde{z}_i = \{ (z_i, \mu_{\tilde{z}_i}(z_i)), \, z_i \in Z_i \subset {\mathbb R}, \, \mu_{\tilde{z}_i}: Z_i \rightarrow [0,1] \}$, with $i = 1, \dotsc, n$. 
A fuzzy vector is represented by a set of pairs 
$$
\tilde{\bf z} = \{ ({\bf z}, \mu_{\tilde{\bf z}}({\bf z})), \, \, {\bf z} \in {\bf Z} \subset {\mathbb R}^n, \, \, \mu_{\tilde{\bf z}}: {\bf Z} \rightarrow [0,1] \}, 
$$
where the universe ${\bf Z}$ is a non-empty subset of the Cartesian product of the one-dimensional universes, i.e. 
${\bf Z}  \subset Z_1 \times \cdots \times Z_n \subset {\mathbb R}^n$, and $\mu_{\tilde{\bf z}}$ is a joint membership function that is obtained from the marginal membership functions $\{ \mu_{\tilde{z}_i} \}_{i=1}^{n}$, based on the interaction of the $n$ fuzzy variables (defined below). 
The set of all fuzzy vectors $\tilde{\bf z}$ on ${\bf Z}$ is denoted by ${\mathcal F}({\bf Z})$.
\end{definition}

Analogous to one-dimensional $\alpha$-cuts we can define $\alpha$-cuts for fuzzy vectors. 
\begin{definition}\label{def_alpha_cut_vec} 
The joint $\alpha$-cuts of the joint membership function $\mu_{\tilde{\bf z}}$ of an $n$-dimensional fuzzy vector $\tilde{\bf z} \in {\mathcal F}({\bf Z})$ are the family $\{ S_{\alpha}^{\tilde{\bf z}} \subset {\bf Z},\, \alpha \in [0,1] \}$ of its joint $\alpha$-level sets:
\begin{equation}\label{joint_alpha-set}
S_0^{\tilde{\bf z}} = {\text{closure}}\{ {\bf z} \in {\bf Z} \, | \, \mu_{\tilde{\bf z}}({\bf z}) >0 \}, \ \ \ \text{and} \ \ \ 
S_{\alpha}^{\tilde{\bf z}} =  \{ {\bf z} \in  {\bf Z} \, | \, \mu_{\tilde{\bf z}}({\bf z}) \ge \alpha \}, \ \ \ \forall \alpha \in (0,1].
\end{equation} 
\end{definition}

Similar to the case of fuzzy variables, we will need the following assumptions:
\begin{itemize}
\item[(A4)] \ \ The universe ${\bf Z} \subset {\mathbb R}^n$ is a bounded, convex set.

\item[(A5)] \ \ The joint membership function is upper semicontinuous.

\item[(A6)] \ \ The joint membership function is quasi-concave.
\end{itemize}
We denote by ${\mathcal F}_c({\bf Z})$ the set of all fuzzy vectors $\tilde{\bf z} \in {\mathcal F}({\bf Z})$ satisfying assumptions (A4)-(A6). 
Note that for a fuzzy vector $\tilde{\bf z} \in {\mathcal F}_c({\bf Z})$, the joint $\alpha$-cuts \eqref{joint_alpha-set} will be compact, convex sets satisfying the inclusion property: 
\begin{equation}\label{inclusion_vectors}
S_{\alpha_2}^{\tilde{\bf z}} \subset S_{\alpha_1}^{\tilde{\bf z}}, \qquad 0 \le \alpha_1 \le \alpha_2 \le 1.
\end{equation}




\noindent
{\bf Interaction.} 
In fuzzy arithmetic it is important to consider the {\it interaction} between fuzzy variables (analogous to the correlation between random variables). 
In general, one can distinguish between three types of interaction and split fuzzy variables into three types: 1) {\it non-interactive} variables; 2) {\it fully interactive} variables; and 3) {\it partially interactive} variables. Interaction can be defined in terms of the notion of $\alpha$-cuts.

\begin{definition}\label{def_noninteractive} 
Consider a fuzzy vector $\tilde{\bf z} \in {\mathcal F}({\bf Z})$ consisting of $n \ge 2$ fuzzy variables $\tilde z_i \in  {\mathcal F}(Z_i)$, with $i = 1, \dotsc, n$, satisfying (A4)-(A6), i.e. $\tilde{\bf z} \in {\mathcal F}_c({\bf Z})$. 
Let $S_{\alpha}^{\tilde{z}_i}$ be the one-dimensional (or marginal) $\alpha$-cut interval corresponding to each fuzzy variable $\tilde z_i$. 
The fuzzy variables $\{ \tilde{z}_i \}_{i=1}^n$ are said to be {\it non-interactive} if their joint $\alpha$-cut $S_{\alpha}^{\tilde{\bf z}}$ is the $n$-dimensional hyperrectangle given by the Cartesian product of $n$ marginal $\alpha$-cuts:
\vskip -.2cm
$$
S_{\alpha}^{\tilde{\bf z}} = S_{\alpha}^{\tilde{z}_1} \times \dotsc \times S_{\alpha}^{\tilde{z}_n} =: \prod_{i=1}^{n} S_{\alpha}^{\tilde{z}_i}, \qquad \ \forall \alpha \in [0,1].
$$
\end{definition}


\begin{definition}\label{def_fullyinteractive} 
The fuzzy variables $\{ \tilde{z}_i \}_{i=1}^n$ in Definition \ref{def_noninteractive} are said to be {\it fully interactive} if their joint $\alpha$-cut $S_{\alpha}^{\tilde{\bf z}}$ is a (possibly non-linear) continuous curve in the hyperrectangle $\prod_{i=1}^n S_{\alpha}^{\tilde{z}_i} \subset {\mathbb R}^n$, satisfying the inclusion property \eqref{inclusion_vectors}.
\end{definition}

It is to be noted that since the joint $\alpha$-cut $S_{\alpha}^{\tilde{\bf z}}$ of fully interactive fuzzy variables is a continuous curve in ${\mathbb R}^n$, there is a bijective mapping between $S_{\alpha}^{\tilde{\bf z}}$ and a one-dimensional closed, bounded interval $I_{\alpha} = [0,L_{\alpha}] \subset {\mathbb R}$, with $L_{\alpha}$ being the Euclidean length of the curve $S_{\alpha}^{\tilde{\bf z}}$. By the arc length parameterization of the curve we can therefore obtain a (possibly non-linear) bijective map $\boldsymbol\varphi_{\alpha} : [0,L_{\alpha}] \rightarrow {\mathbb R}^n$ so that $\boldsymbol\varphi_{\alpha}(s) \in S_{\alpha}^{\tilde{\bf z}}$ for each arc length parameter $s \in I_{\alpha} = [0,L_{\alpha}]$. Such mapping facilitates practical fuzzy computations; see Section \ref{sec:numerics}.

Clearly, unlike the case of non-interactive fuzzy variables for which the inclusion property \eqref{inclusion_vectors} is automatically satisfied, in the case of fully-interactive variables the inclusion property must be imposed, because not every continuous curve in the hyperrectangle satisfies this property. 
A particular type of full interaction that can be easily handled in practical computations may be considered by setting the joint $\alpha$-cut to be the polygonal (i.e. continuous and piecewise linear) curve, given recursively by 
\vskip -.1cm
$$
S_1^{\tilde{\bf z}} = diag( \prod_{i=1}^n S_1^{\tilde{z}_i}),
$$
\vskip -.3cm
$$
S_{\alpha_{j}}^{\tilde{\bf z}} = S_{\alpha_{j+1}}^{\tilde{\bf z}} \bigcup \ diag( \prod_{i=1}^n [ S_{\alpha_{j}}^{\tilde{z}_i} \setminus S_{\alpha_{j+1}}^{\tilde{z}_i} ]_l )
\bigcup \ diag( \prod_{i=1}^n [ S_{\alpha_{j}}^{\tilde{z}_i} \setminus S_{\alpha_{j+1}}^{\tilde{z}_i} ]_r ), 
 \ \  0 \le \alpha_j < \alpha_{j+1} \le 1.
$$
Here, $diag(\prod_{i=1}^n S^{i})$ with $S^i=[\ubar{S}^i, \bar{S}^i]$ denotes the main space diagonal of the hyperrectangle $S=\prod_{i=1}^n S^{i}$, i.e. the line segment between the vertices $(\ubar{S}^1, \dotsc, \ubar{S}^n)$ and $(\bar{S}^1, \dotsc, \bar{S}^n)$, and the hyperrectangles $[ S_{\alpha_{j}}^{\tilde{z}_i} \setminus S_{\alpha_{j+1}}^{\tilde{z}_i} ]_l$ and $[ S_{\alpha_{j}}^{\tilde{z}_i} \setminus S_{\alpha_{j+1}}^{\tilde{z}_i} ]_r$ are the left and right portions of the set $S_{\alpha_{j}}^{\tilde{z}_i} \setminus S_{\alpha_{j+1}}^{\tilde{z}_i}$, respectively; see Figure \ref{interactive_alpha_cut}. 
We notice that this setting ensures that the inclusion property \eqref{inclusion_vectors} holds.
\begin{figure}[!h]
\vspace{-.2cm}
  \begin{center}
        \includegraphics[width=0.4\linewidth]{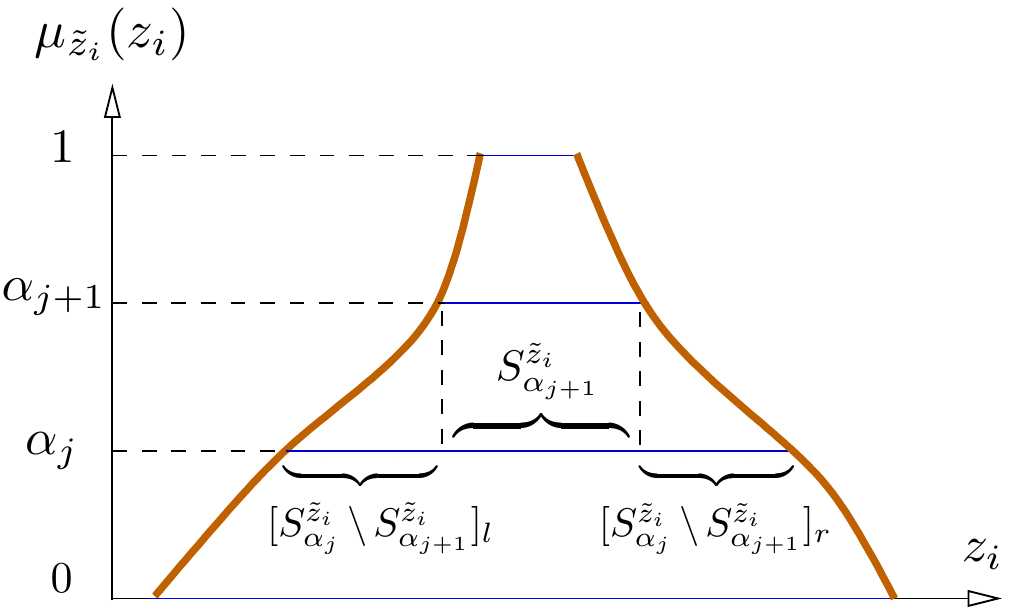}
\caption{The intervals used in the definition of the joint $\alpha$-cuts of fully interactive fuzzy variables by polygonal nested curves. We have 
$S_{\alpha_{j}}^{\tilde{z}_i} = S_{\alpha_{j+1}}^{\tilde{z}_i} \bigcup \
[S_{\alpha_{j}}^{\tilde{z}_i} \setminus S_{\alpha_{j+1}}^{\tilde{z}_i} ]_l 
\bigcup \
[S_{\alpha_{j}}^{\tilde{z}_i} \setminus S_{\alpha_{j+1}}^{\tilde{z}_i} ]_r$ with $\alpha_j < \alpha_{j+1}$.}
\label{interactive_alpha_cut}
  \end{center}
  \vskip -.1cm
\end{figure}

\begin{definition}\label{def_partiallyinteractive} 
The fuzzy variables $\{ \tilde{z}_i \}_{i=1}^n$ in Definition \ref{def_noninteractive} are said to be {\it partially interactive} if they are neither non-interactive nor fully interactive. 
\end{definition}

Due to the assumptions (A4)-(A6), Definition \ref{def_partiallyinteractive} implies that the joint $\alpha$-cut $S_{\alpha}^{\tilde{\bf z}}$ of partially interactive fuzzy variables is a strict subset of the Cartesian product of the marginal $\alpha$-cuts, that is, $S_{\alpha}^{\tilde{\bf z}} \subsetneq \prod_{i=1}^n S_{\alpha}^{\tilde{z}_i}$, $\, \forall \alpha \in [0,1]$. Moreover, it cannot be mapped into a one-dimensional interval. In practice the joint $\alpha$-cuts of partially interactive variables are geometrically more complicated than those of non-interactive and fully interactive variables. Efficient fuzzy computation in the case of partial interaction is a challenging task due to the need for solving constrained optimization problems over complicated joint $\alpha$-cuts. We refer to \cite{Scheerlinck_etal:14} for the treatment of particular types of partial interaction using triangular norms. 

Importantly, partial interaction does not often occur in hybrid fuzzy-stochastic modeling, and hence we may not need to treat this case in a hybrid fuzzy-stochastic framework. Full interaction may however occur in hybrid fuzzy-stochastic models. An example is when the uncertain parameter is characterized by a random variable with fuzzy statistical moments using a set of measurement data. Since the moments are directly related to each other, i.e. higher moments are obtained from lower moments, the fuzzy moments may be fully interactive; see Section \ref{sec:numerics2}. Another example, which may occur even in a pure fuzzy framework, is when we perform mathematical operations on two functions with the same fuzzy arguments; see Section \ref{sec:F3} for the definition of fuzzy functions. In this case the outputs of the two fuzzy functions are fully interactive.



Although in a hybrid framework we often face non-interactive and/or fully interactive fuzzy variables, we note that the mathematical definitions and results in the present paper are independent of the type of interaction between fuzzy variables. 
In the rest of the paper whenever the type of interaction between fuzzy variables is not specified, the fuzzy variables are understood as general fuzzy variables in ${\mathcal F}({\bf Z})$.

\subsection{Fuzzy functions}
\label{sec:F3}

A fuzzy function is a generalization of the concept of a classical function. A classical function is a mapping from its domain of definition into its range. 
There are various generalizations in the literature on fuzzy calculus; see e.g. \cite{Dubois_Prade:1980, Moller_Beer:2004} and the references therein. 
Here, we consider only two cases: 1) a crisp map with fuzzy arguments, and 2) a crisp map with both fuzzy and nonfuzzy arguments. 

%
\begin{definition}\label{fuzzy_fcn_def} 
Consider a function $u: {\bf Z} \rightarrow V$ mapping every element of its domain ${\bf Z}\subset {\mathbb R}^n$ into an element of its range $V = u({\bf Z}) \subset {\mathbb R}$. 
%
Let further $\tilde{\bf z} \in {\mathcal F}({\bf Z})$ be a fuzzy vector with a joint membership function $\mu_{\tilde{\bf z}}: {\bf Z} \rightarrow [0,1]$. A function $u$ of $\tilde{\bf z}$, referred to as a fuzzy function, is a mapping $u: {\mathcal F}({\bf Z}) \rightarrow {\mathcal F}(V)$, 
so that 
$$
\tilde u := u(\tilde{\bf z}) = \{ ( v, \mu_{\tilde{u}}(v) ), \, v \in V=u({\bf Z}) \} \in {\mathcal F}(V), \qquad \tilde{\bf z} \in {\mathcal F}({\bf Z}),
$$ 
is a fuzzy variable with the membership function $\mu_{\tilde{u}}$ given by the {\it generalized extension principle} \cite{Fuller:04,Fuller:14}:
\begin{equation}\label{joint_output}
\mu_{\tilde{u}}(v) =\left\{ \begin{array}{l l}
\sup_{{\bf z}= u^{-1}(v)}  \, \mu_{\tilde{\bf z}}({\bf z}) & \qquad u^{-1}(v) \neq \emptyset \\
0 &  \qquad u^{-1}(v) = \emptyset
\end{array} \right.,
\qquad \ \ \ \forall v \in V.
\end{equation}
Here, $u^{-1}(v)$ is the inverse image of $v = u({\bf z}) \in V$ 
and $\emptyset$ is the empty set. 
\end{definition}
%

Crucially, the generalized extension principle \eqref{joint_output} uses the general form of the input joint membership function $\mu_{\tilde{\bf z}}$ and hence is valid for both non-interactive and interactive input fuzzy variables. 

\begin{remark}
This notion of fuzzy functions was originally introduced by Zadeh \cite{Zadeh:65,Zadeh:75} for non-interactive input fuzzy variables, where $\mu_{\tilde{\bf z}}$ is given by 
the minimum of the marginal memberships 
$\mu_{\tilde{\bf z}}({\bf z}) = \min (\mu_{\tilde{z}_1}(z_1), \dotsc, \mu_{\tilde{z}_n}(z_n)), \ \forall {\bf z} = (z_1, \dotsc, z_n) \in {\bf Z}$, 
and the generalized extension principle \eqref{joint_output} reduces to Zadeh's sup-min extension principle. 
In \cite{Dubois_Prade:1980} this type of mapping is called ``fuzzy extension of a nonfuzzy function''. 
\end{remark}

%
%
%
%

In order to define fuzzy and fuzzy-stochastic fields that appear in the study of fuzzy-stochastic PDEs, we further need to consider crisp maps with both fuzzy arguments and nonfuzzy arguments, such as spatial, temporal, and random variables. For this purpose, we introduce the notion of fuzzy Banach space-valued functions.
%

\begin{definition}\label{Sobolev_fuzzy_fcn_def} 
Consider a real-valued function $u: {\bf P} \times {\bf Z} \rightarrow V$, mapping every element of its domain ${\bf P} \times {\bf Z}$, with ${\bf P} \subset {\mathbb R}^p$ and ${\bf Z} \subset {\mathbb R}^n$, to an element of its range $V = u({\bf P} \times {\bf Z}) \subset {\mathbb R}$. 
%
%
Associate with $u$ a mapping $u: {\bf Z} \rightarrow B({\bf P})$, defined by 
$$
[u({\bf z})] ({\bf p}) := u({\bf p}, {\bf z}), \qquad {\bf p} \in {\bf P}, \ \  {\bf z} \in {\bf Z}, 
$$
where $B({\bf P})$ is a metric space of functions on ${\bf P}$. 
Let $\tilde{\bf z} \in {\mathcal F}({\bf Z})$ be a fuzzy vector with a joint membership function $\mu_{\tilde{\bf z}}: {\bf Z} \rightarrow [0,1]$. 
A function $u$ of ${\bf p} \in {\bf P}$ and $\tilde{\bf z}\in {\mathcal F}({\bf Z})$, 
written as
\begin{equation}\label{Fuzzy_Sobolev_fcn}
[u (\tilde{\bf z})]  ({\bf p}) := u({\bf p}, \tilde{\bf z}), \qquad  {\bf p} \in {\bf P}, \qquad \tilde{\bf z}\in {\mathcal F}({\bf Z}),
\end{equation}
is defined by a family of fuzzy variables $\{ \tilde{u}({\bf p}),  {\bf p} \in {\bf P} \}$. Each element of this set is a fuzzy variable $\tilde{u}({\bf p}) := [u (\tilde{\bf z})]  ({\bf p})$ 
corresponding to a fixed ${\bf p} \in {\bf P}$, given by 
$$
\tilde{u} ({\bf p}) =
 \{ \left( v , \, \mu_{\tilde{u}({\bf p})}( v ) \right), \, v \in u( {\bf p} ,{\mathbf Z}) \subset {\mathbb R}\}, \qquad
\forall {\bf p} \in {\bf P},
$$
with the membership function $\mu_{\tilde{u}({\bf p})}$ given by the generalized extension principle \eqref{joint_output}. 
The restriction of \eqref{Fuzzy_Sobolev_fcn} to ${\bf Z}$ is $u: {\bf Z} \rightarrow B({\bf P})$. The function \eqref{Fuzzy_Sobolev_fcn} may then be viewed as a fuzzy function taking values in a metric space. If the metric space $B({\bf P})$ is a Banach space, we call such a function a fuzzy Banach space-valued function, denoted by the mapping 
$u: {\mathcal F}({\bf Z}) \rightarrow {\mathcal F}(B({\bf P}))$.
\end{definition}



We note that Sobolev and Hilbert spaces of functions are specific Banach spaces that we deal with in the context of fuzzy-stochastic PDEs, where functions with temporal/spatial, random, and fuzzy arguments appear as the coefficients, data, and solutions to the PDE problems.


\begin{remark} 
A fuzzy Banach space-valued function is closely related to a fuzzy map with nonfuzzy arguments discussed in \cite{Moller_Beer:2004}. A fuzzy mapping on the nonfuzzy variables ${\bf p}$, denoted by $\tilde{u}({\bf p})$, may be formulated as a crisp mapping $u({\bf p}, \tilde{\bf z})$ on the nonfuzzy variables ${\bf p}$ and a fuzzy vector $\tilde{\bf z}$ referred to as ``fuzzy bunch parameters'', i.e.  $\tilde{u}({\bf p}) = u({\bf p}, \tilde{\bf z})$. In \cite{Moller_Beer:2004} this is called the bunch parameter representation of fuzzy functions. Fuzzy Banach space-valued functions are also related to ``fuzzifying functions'' discussed in \cite{Dubois_Prade:1980}.
\end{remark}


\noindent
{\bf Computation of fuzzy functions.} 
The computation of a fuzzy function $\tilde{u}$ amounts to the computation of its output membership function $\mu_{\tilde{u}}$. 
Computing $\mu_{\tilde{u}}(v)$ for all $v = u({\bf z}) \in V$ by a direct application of the generalized extension principle \eqref{joint_output} can be quite complicated and numerically cumbersome, as there is no efficient method to evaluate the supremum of $\mu_{\tilde{\bf z}}({\bf z})$ over all ${\bf z}$ for which $u({\bf z}) = v$. The computations can substantially be simplified using the $\alpha$-cut representation of $\mu_{\tilde{u}}$, thanks to the following important result, referred to as the {\it function-set identity} \cite{Fuller:04,Fuller:14}, extending the earlier work of Nguyen \cite{Nguyen:78}.
\begin{theorem} \label{fuzzy_set_thm} (Function-set identity \cite{Fuller:04,Fuller:14}) 
Let $\tilde{\bf z} \in {\mathcal F}({\bf Z})$ be a fuzzy vector with a joint membership function $\mu_{\tilde{\bf z}}: {\bf Z} \rightarrow [0,1]$ and corresponding joint $\alpha$-cuts $S_{\alpha}^{\tilde{\bf z}}$, satisfying the assumptions (A4)-(A6), i.e. $\tilde{\bf z} \in {\mathcal F}_c({\bf Z})$. 
Let further $u: {\bf Z} \rightarrow V$ be a continuous map, where $V = u({\bf Z}) \subset {\mathbb R}$. Then the $\alpha$-cuts $S_{\alpha}^{\tilde{u}}$ corresponding to the output membership function $\mu_{\tilde{u}}$ of the fuzzy function $\tilde u = u(\tilde{\bf z}) \in {\mathcal F}(V)$ are given by:
\begin{equation}\label{function_set1}
S_{\alpha}^{\tilde{u}} = u(S_{\alpha}^{\tilde{\bf z}}) = [ \min_{{\bf z} \in S_{\alpha}^{\tilde{\bf z}}} u({\bf z}), \ \max_{{\bf z} \in S_{\alpha}^{\tilde{\bf z}}} u({\bf z}) ], \qquad \forall \alpha \in [0,1].
\end{equation}
\end{theorem}

It is to be noted that two conditions must be satisfied for \eqref{function_set1} to hold: 
1) the map $u: {\bf Z} \rightarrow V$ is continuous, and 
2) the fuzzy input vector $\tilde{\bf z} \in {\mathcal F}({\bf Z})$ satisfies the assumptions (A4)-(A6). Under these two conditions, the $\alpha$-cuts $S_{\alpha}^{\tilde{u}}$ will be compact intervals given by \eqref{function_set1}. The continuity assumption holds when, for instance, the function $\tilde{u}$ is the solution to a differential equation under appropriate assumptions on the data. This important observation will be later utilized for the analysis and computation of fuzzy-stochastic PDEs in this paper.

Crucially, Theorem \ref{fuzzy_set_thm} allows us to decompose fuzzy computations into several interval computations. 
%
Motivated by this, we present a numerical approach, outlined in Algorithm \ref{ALG_fuzzy_functions}, for computing fuzzy functions.



\begin{algorithm}[!ht]
\caption{Computation of fuzzy functions}
\label{ALG_fuzzy_functions}
\begin{algorithmic} 
\medskip
\STATE {\bf 0.} Given a fuzzy vector $\tilde{\bf z} \in {\mathcal F}_c({\bf Z})$ and a continuous map $u: {\bf Z} \rightarrow V$, where $V = u({\bf Z}) \subset {\mathbb R}$, 
the output membership function of the fuzzy function $\tilde u = u(\tilde{\bf z}) \in {\mathcal F}(V)$ is computed as follows.


\medskip
\STATE {\bf 1.} {\it Interaction}: Find the input joint $\alpha$-cut $S_{\alpha}^{\tilde{\bf z}}$ for a fixed $\alpha \in [0,1]$ based on the interaction between the input fuzzy variables.


\medskip
\STATE {\bf 2.} {\it Optimization}: Obtain the output $\alpha$-cut $S_{\alpha}^{\tilde{u}} = [\ubar{u},\bar{u}]$ by solving two global optimization problems: 
$
\ 
\ubar{u} : = \displaystyle\min_{{\bf z} \in S_{\alpha}^{\tilde{\bf z}}} u({\bf z}) \
$
and 
$ 
\ 
\bar{u} : = \displaystyle\max_{{\bf z} \in S_{\alpha}^{\tilde{\bf z}}} u({\bf z}).
$
 
\STATE {\bf 3.} Repeat steps {\bf 1}-{\bf 2} for various $\alpha$ and form the output membership function $\mu_{\tilde{u}}$.

%

\end{algorithmic}
\end{algorithm}

The optimization problems in step 2 can be numerically solved for instance by iterative methods; see e.g. \cite{Incomplete_search:03,Neumaier:04,optimization_fuzzy:12,Scheerlinck_etal:14}. The choice of the method would depend on the dimension and the complexity of $S_{\alpha}^{\tilde{\bf z}}$ and the regularity of $u$ with respect to ${\bf z}$. 

Similarly, the output membership function of a fuzzy Banach space-valued function (see Definition \ref{Sobolev_fuzzy_fcn_def}) may be computed by Algorithm \ref{ALG_fuzzy_functions} pointwise in ${\bf p} \in {\bf P}$. 

\subsection{Fuzzy fields}
\label{sec:F4}

A scalar fuzzy field is a particular type of a fuzzy Banach space-valued function. It is a crisp map with spatial variables and a fuzzy vector as arguments generating an infinite set of fuzzy variables.  
\begin{definition}\label{fuzzy_field_def} 
Let $D \subset {\mathbb R}^d$ be a compact spatial domain, with $d=1,2,3$, and consider a vector of spatial variables ${\bf x} \in D$. Let further $\tilde{\bf z} \in {\mathcal F}({\bf Z})$ be a fuzzy vector on ${\bf Z} \subset {\mathbb R}^n$. A scalar fuzzy field, written as $\tilde u ({\bf x}) = u({\bf x}, \tilde{\bf z}), \, \forall {\bf x} \in D$, is a fuzzy Banach space-valued function 
$u: {\mathcal F}({\bf Z}) \rightarrow  {\mathcal F}(B(D))$, where $B(D)$ is a Banach space on $D$.
\end{definition}

A typical example of $B(D)$ in the context of PDEs is the Hilbert space of functions whose weak derivatives up to order $s \ge 0$ are square integrable, denoted by $H^s(D)$.



\subsection{Fuzzy random variables}
\label{sec:FS1}

Fuzzy random variables were introduced by Kwakernaak \cite{Kwakernaak:78,Kwakernaak:79} as random variables whose values are not real numbers but fuzzy sets. A fuzzy random variable in the sense of Kwakernaak is viewed as a fuzzy perception of an underlying (or {\it original}) real-valued random variable defined on a probability space $(\Omega, \Sigma, P)$, where $\Omega$ is a sample space, $\Sigma$ is a $sigma$-field on $\Omega$, and $P$ is a probability measure assigned to each measurable subset of $\Omega$ and satisfying Kolmogorov's axioms; see e.g. \cite{Grigoriu}. The definition of fuzzy random variables in the sense of Kwakernaak was later formalized in a clear way by Kruse and Meyer \cite{KruseMeyer:87}.


\begin{definition}\label{fuzzy_stock_var_def} 
Let $(\Omega, \Sigma, P)$ be a probability space. Let ${\mathcal F}_c({\mathbb R})$ be the space of all fuzzy variables in ${\mathcal F}({\mathbb R})$ satisfying assumptions (A1)-(A3). Let also ${\mathcal M}_c({\mathbb R})$ be the space of upper semicontinuous, quasi-concave functions ${\mathbb R} \rightarrow [0,1]$ with compact closure of their support; that is, the space of membership functions corresponding to the fuzzy variables in ${\mathcal F}_c({\mathbb R})$.  
A mapping $\tilde{y}: \Omega \rightarrow {\mathcal F}_c({\mathbb R})$, or equivalently $\mu_{\tilde{y}} : \Omega \rightarrow {\mathcal M}_c({\mathbb R})$, is said to be a fuzzy random variable associated with $(\Omega, \Sigma, P)$ if 
$\forall \alpha \in (0,1]$ both $y_{\alpha}^L : \Omega \rightarrow {\mathbb R}$ and $y_{\alpha}^R : \Omega \rightarrow {\mathbb R}$, defined as
\begin{align*}
y_{\alpha}^L (\omega)  &:= \min \{ y \in {\mathbb R} \, | \, [\mu_{\tilde{y}}(\omega)] (y) \ge \alpha \}, \, \qquad
\forall \omega \in \Omega,\\
y_{\alpha}^R (\omega)  &:= \max \{ y \in {\mathbb R} \, | \, [\mu_{\tilde{y}}(\omega)] (y) \ge \alpha \}, \qquad
\forall \omega \in \Omega,
\end{align*}
are real-valued random variables on $(\Omega, \Sigma, P)$.



\end{definition}

It is to be noted that Definition \ref{fuzzy_stock_var_def} implies that every realization of a fuzzy random variable $\tilde{y}(\omega)$, for some $\omega \in \Omega$, is a fuzzy variable, given by a set of pairs 
$$
\tilde{y}(\omega) = \{ (y, [\mu_{\tilde{y}}(\omega)] (y)), \, \, y \in  {\mathbb R}, \, \, \mu_{\tilde{y}}(\omega) \in {\mathcal M}_c({\mathbb R})  \}, \qquad \forall \omega \in \Omega.
$$
%
Or equivalently, we associate with each $\omega \in \Omega$ not a real number as in the case of an ordinary random variable, but a membership function $\mu_{\tilde{y}}(\omega)$ which is an element of ${\mathcal M}_c({\mathbb R})$. We also note that the requirement that $y_{\alpha}^L$ and $y_{\alpha}^R$ are real-valued random variables defined on $(\Omega, \Sigma, P)$ imposes a measurability condition on the map $\tilde{y}$.

The distribution of a fuzzy random variable $\tilde{y}$ in the sense of Kwakernaak can then be derived from a fuzzy perception of the distribution of the ``original'' random variables through Zadeh's extension principle. 
We first note that by Definition \ref{fuzzy_stock_var_def} there may be more than one original random variable on $(\Omega, \Sigma, P)$ corresponding to a given fuzzy random variable $\tilde{y}$ associated with $(\Omega, \Sigma, P)$. Let $\text{Orig}(\tilde{y})$ denote the set of all potential originals of $\tilde{y}$. Assume that the probability measure $P$ is absolutely continuous with respect to Lebesgue measure, and let $y$ and $F_y$ denote an original random variable in $\text{Orig}(\tilde{y})$ and its corresponding CDF, respectively. Then we can formulate the definition of a fuzzy-valued distribution as follows.

\begin{definition}\label{fuzzy_CDF_def_new} 
The fuzzy-valued distribution function of a fuzzy random variable $\tilde{y}$ associated with $(\Omega, \Sigma, P)$ is a mapping 
$\tilde{F}_{\tilde{y}}: {\mathbb R} \rightarrow {\mathcal M}_c({\mathbb R})$, given by
\begin{equation}\label{Fuzzy_CDF_new}
[\tilde{F}_{\tilde{y}} (y_0) ] (p) =\left\{ \begin{array}{l l}
\sup\limits_{\substack{y \in \text{Orig}(\tilde{y}) \\ F_y(y_0) = p}}  \  \inf\limits_{\omega \in \Omega} \  [\mu_{\tilde{y}}(\omega)] (y(\omega)) & \quad \text{if} \ p \in [0,1]\\
\ \ \ 0 &  \quad \text{otherwise}
\end{array} \right., \qquad \forall y_0 \in {\mathbb R},
\end{equation}
where $\text{Orig}(\tilde{y})$ denotes the set of all potential originals of $\tilde{y}$. 
\end{definition}

We note that Definition \ref{fuzzy_CDF_def_new} implies that the distribution function $\tilde{F}_{\tilde{y}} (y_0)$ at a fixed point $y_0 \in {\mathbb R}$ is a fuzzy variable in ${\mathcal F}_c({\mathbb R})$, characterized by a membership function in ${\mathcal M}_c({\mathbb R})$ supported on the interval $[0,1]$. 

Clearly, a direct application of \eqref{Fuzzy_CDF_new} may be impossible in practice, as there may be no efficient method to evaluate the supremum over all original distributions. For this purpose, we consider a simpler case, and yet general enough for many practical applications, when the original random variable and its corresponding distribution are known. In this case, the fuzzy random variable can be viewed as a random variable with fuzzy parameters. For instance, a normal fuzzy random variable $\tilde{y} \sim {\mathcal N}(\tilde{\theta}_1, \tilde{\theta}_2^2)$ can be viewed as a normal random variable with two fuzzy parameters: a fuzzy mean $\tilde{\theta}_1$ and a fuzzy standard deviation $\tilde{\theta}_2$. For such type of fuzzy random variables, thanks to the continuity of the original CDF, we can use Theorem \ref{fuzzy_set_thm} and formulate fuzzy CDFs as follows. Let $F_y(y_0;\boldsymbol\theta)$ denote the (parametrized) CDF of the original random variable with real-valued parameters $\boldsymbol\theta \in {\mathbb R}^n$ evaluated at $y_0 \in {\mathbb R}$. For instance, the original distribution of a normal random variable with parameters $(\theta_1, \theta_2)$ reads
$$
F_y(y_0;\boldsymbol\theta) = \frac{1}{\sqrt{2 \pi} \theta_2} \int_{- \infty}^{y_0} e^{\frac{-(\tau - \theta_1)^2}{2 \, \theta_2^2}} \, d\tau, \qquad y_0 \in {\mathbb R}, \qquad \boldsymbol\theta = (\theta_1, \theta_2) \in {\mathbb R}^2.
$$
%

\begin{definition}\label{fuzzy_cdf_typeI_def} 
Consider a fuzzy random variable $\tilde{y}$ consisting of a known original random variable with $n$ fuzzy parameters $\tilde{\boldsymbol\theta} \in {\mathcal F}_c({\bf Z})$, with ${\bf Z} \subset {\mathbb R}^n$. Let $S_{\alpha}^{\tilde{\boldsymbol\theta}} \subset {\mathbb R}^n$ be the joint $\alpha$-cut of $\tilde{\boldsymbol\theta}$ and $F_y(y_0; \boldsymbol\theta)$ be the parametrized CDF of the original random variable, with $y_0 \in {\mathbb R}$ and $\boldsymbol\theta \in {\mathbb R}^n$. 
%
At any fixed $\alpha$-level, let $F_{\alpha}^{L}$ and $ F_{\alpha}^{R}$ be the extrema of the family of parameterized CDFs over the joint $\alpha$-cut $S_{\alpha}^{\tilde{\boldsymbol\theta}}$, referred to as the left (upper) and right (lower) bounds: 
\begingroup
\setlength\abovedisplayskip{5pt}
\setlength\belowdisplayskip{5pt}
\begin{equation}\label{envelopes2}
F_{\alpha}^{L}(y_0) = \max_{\boldsymbol\theta \in S_{\alpha}^{\tilde{\boldsymbol\theta}}} F_y(y_0;  \boldsymbol\theta), \qquad
F_{\alpha}^{R}(y_0) = \min_{\boldsymbol\theta \in S_{\alpha}^{\tilde{\boldsymbol\theta}}} F_y(y_0;  \boldsymbol\theta), \quad
\forall \alpha \in [0,1], \ \ \forall y_0 \in {\mathbb R}.
\end{equation}
\endgroup
The fuzzy CDF of $\tilde{y}$, evaluated at $y_0 \in {\mathbb R}$ and denoted by $\tilde{F}_{\tilde{y}}(y_0) = F_y(y_0; \tilde{\boldsymbol\theta})$, is then defined by a nested set of left and right bounds at different $\alpha$-levels:
$$
\tilde{F}_{\tilde{y}}(y_0) = F_y(y_0; \tilde{\boldsymbol\theta}) = \big\{ \bigl( F_{\alpha}^{L}(y_0), F_{\alpha}^{R}(y_0) \bigr), \ \alpha \in [0,1] \big\}, \qquad \forall y_0 \in {\mathbb R}.
$$
\end{definition}

\noindent
{\bf Interpretation of fuzzy CDFs.} 
We note that the fuzzy CDF in Definition \ref{fuzzy_cdf_typeI_def} is a particular type of the fuzzy CDF in Definition \ref{fuzzy_CDF_def_new}. In general, for any fixed $\alpha$-level, the set of all left and right bounds of the distribution, e.g. those given in \eqref{envelopes2}, corresponding to all points $y_0 \in {\mathbb R}$ will constitute two left and right {\it envelopes} forming a p-box. It is important to note that the left and right envelopes are not necessarily two single CDFs. In fact, for different values of $y_0 \in {\mathbb R}$, there may exist different maximizers and/or minimizers. Hence, different distributions on different regions may constitute the two envelopes. 
A fuzzy CDF can indeed be thought of as a nested set of p-boxes at different levels of possibility (corresponding to different $\alpha$-levels); see also the numerical examples in Section \ref{sec:numerics}. Hence, fuzzy CDFs provide a far more comprehensive representation of uncertainty, compared to a class of imprecise probabilistic models such as p-boxes \cite{Weichselberger:2000}, coherent lower and upper previsions \cite{Walley:91,Walley:2000}, and optimal UQ \cite{OUQ:2013} which provide only crisp lower and upper bounds from a set of admissible distributions.

\begin{remark} (Fuzzy random variables as random elements) 
There are various definitions of fuzzy random variables utilizing the well-developed concept of random elements in Fr{\'e}chet's sense \cite{Frechet:1948}. For instance, a fuzzy random variable in the sense of Puri and Ralescu \cite{Puri_Ralescu:1986} is defined as a random compact set, that is, a Borel-measurable mapping with respect to the Borel $\sigma$-field generated by the topology associated with the Haussdroff metric on the space of nonempty compact subsets of ${\mathbb R}^n$. 
Colubi et. al. \cite{Colubi:2001,Colubi:2002} define fuzzy random variables as random upper semi-continuous functions, that is, Borel-measurable mappings with respect to the Borel $\sigma$-field generated by the topology associated with the Skorokhod metric on the space of upper semicontinuous functions ${\mathbb R} \rightarrow [0, 1]$ with compact closure of their support. Both definitions generalize Kwaakernak's definition by relaxing the convexity assumption and extend it to multiple dimensions, introducing the notion of fuzzy random vectors. However, for many practical developments, these general definitions are often particularized to the one-dimensional case and to convex fuzzy variables. Under these particularizations, the general definitions reduce to Kwaakernak's definition; see e.g. \cite{Blanco-Fernandez_Etal:2014,Gil:2018}.
\end{remark}

\subsection{Fuzzy-stochastic functions}
\label{sec:FS2}

A fuzzy-stochastic function is a particular type of a fuzzy Banach space-valued function. It is a crisp map with a random vector and a fuzzy vector as arguments generating an output fuzzy random variable. 
\begin{definition}\label{fuzzy_stoch_fcn_def} 
Let ${\bf y} \in \Gamma \subset {\mathbb R}^m$ be a random vector with a bounded joint probability distribution function (PDF). Let $\tilde{\bf z} \in {\mathcal F}_c({\bf Z})$ be a fuzzy vector on ${\bf Z} \subset {\mathbb R}^n$. 
A fuzzy-stochastic function, written as $\tilde{u} ({\bf y}) = u({\bf y}, \tilde{\bf z}), \, \forall {\bf y} \in \Gamma$, is a fuzzy Banach space-valued function 
$u: {\mathcal F}({\bf Z}) \rightarrow  {\mathcal F}(B(\Gamma))$, with $B(\Gamma)$ being a Banach space of random functions on $\Gamma$. 
\end{definition}

A typical example of $B(\Gamma)$ is the space of random functions with bounded second moments, denoted by $L_{\pi}^2(\Gamma)$, where $\pi: \Gamma \rightarrow {\mathbb R}_+$ is the bounded joint PDF of ${\bf y}$.


We note that the output of a fuzzy-stochastic function $\tilde{u}$ is a particular type of fuzzy random variable for which the distribution of the ``original'' random variable is known through the joint PDF $\pi$ of the input random vector. Hence, the fuzzy CDF of $\tilde{u}$, denoted by $\tilde{F}_{\tilde{u}}$, can simply be defined through $\pi$ as in Definition \ref{fuzzy_cdf_typeI_def}.

\begin{definition}\label{fuzzy_cdf_typeII_def} 
Consider a fuzzy random variable $\tilde{u}$, being the output of a fuzzy-stochastic function defined in Definition \ref{fuzzy_stoch_fcn_def}. 
For every fixed ${\bf z} \in S_{\alpha}^{\tilde{\bf z}}$, the parameterized CDF of the corresponding original random variable, evaluated at any point $u_0 \in {\mathbb R}$, is determined by the PDF of the input random vector $\pi = \pi({\bf y})$ as
$$
F_u(u_0; {\bf z}) = \int_{\{ {\boldsymbol\tau}: u(\boldsymbol\tau, {\bf z}) \le u_0\}} \pi(\boldsymbol\tau) \, d\boldsymbol\tau.
$$
At any fixed $\alpha$-level, let $F_{\alpha}^{L}$ and $ F_{\alpha}^{R}$ be the extrema of the family of parameterized CDFs over the joint $\alpha$-cut $S_{\alpha}^{\tilde{\bf z}}$:  
$$
F_{\alpha}^{L}(u_0) = \max_{{\bf z} \in S_{\alpha}^{\tilde{\bf z}}} F_u(u_0;  {\bf z}), \qquad
F_{\alpha}^{R}(u_0) = \min_{{\bf z} \in S_{\alpha}^{\tilde{\bf z}}} F_u(u_0;  {\bf z}), \quad
\forall \alpha \in [0,1], \ \ \forall u_0 \in {\mathbb R}.
$$
%
The fuzzy CDF of $\tilde{u}$, evaluated at $u_0 \in {\mathbb R}$ and denoted by $\tilde{F}_{\tilde{u}}(u_0) = F_u(u_0; \tilde{\bf z})$, is then defined by a nested set of left and right bounds at different $\alpha$-levels:
$$
\tilde{F}_{\tilde{u}}(u_0) = F_u(u_0; \tilde{\bf z}) = \big\{ \bigl( F_{\alpha}^{L}(u_0), F_{\alpha}^{R}(u_0) \bigr), \ \alpha \in [0,1] \big\}, \qquad \forall  u_0 \in {\mathbb R}.
$$

\end{definition}

\noindent
{\bf Computation of fuzzy-stochastic functions.} Let $\tilde{u} ({\bf y}) = u({\bf y}, \tilde{\bf z})$ be a fuzzy-stochastic function. Assume that we want to compute a fuzzy QoI, denoted by $\tilde{Q}$, given in terms of $\tilde{u} ({\bf y})$. Three important examples of $\tilde{Q}$ include:
\medskip
\begin{itemize}[topsep=1pt,leftmargin=2.6\labelsep]
\setlength\itemsep{-.25em}
\item $\tilde{Q} = {\mathbb E}[ u^r({\bf y}, \tilde{\bf z}) ]$: the $r$-th fuzzy moment of $\tilde{u} ({\bf y})$, where $r$ is a positive integer. 

\medskip
\item $\tilde{Q} = \tilde{F}_{\tilde{u}}(u_0) = {\mathbb E}[ {\mathbb I}_{[u({\bf y}, \tilde{\bf z}) \le u_0]} ]$: the fuzzy CDF of $\tilde{u} ({\bf y})$ evaluated at a fixed point $u_0 \in {\mathbb R}$, where ${\mathbb I}_{[\cdot]}$ is the indicator function taking the value 1 or 0 if the event $[\cdot]$ is ``true'' or ``false'', respectively.

\medskip
\item $\tilde{Q} = {\mathbb E}[{\mathbb I}_{[g(u({\bf y},\tilde{\bf z})) \le 0]}]$: the fuzzy failure probability assuming that failure occurs when $g(\tilde{u}({\bf y})) \le 0$, where $g$ is a differential and/or integral operator on $\tilde{u}({\bf y})$. 

\end{itemize}

\medskip
\noindent
Each of the above fuzzy QoIs is the expectation of a fuzzy-stochastic function, say $\tilde{Q} = {\mathbb E}[q({\bf y}, \tilde{\bf z})]$, where $q= u^r({\bf y}, \tilde{\bf z})$ in the first example, $q = {\mathbb I}_{[u({\bf y}, \tilde{\bf z}) \le u_0]}$ in the second example, and $q={\mathbb I}_{[g(u({\bf y},\tilde{\bf z})) \le 0]}$ in the third example above. 
Algorithm \ref{ALG_fuzzy_random_functions} outlines a numerical approach for computing $\tilde{Q}$. 
%
%
%
\begin{algorithm}
\caption{Computation of fuzzy-stochastic functions}
\label{ALG_fuzzy_random_functions}
\begin{algorithmic} 
\medskip
\STATE {\bf 0.} 
Given a random vector ${\bf y} \in \Gamma$, a fuzzy vector $\tilde{\bf z} \in {\mathcal F}_c({\bf Z})$, and a fuzzy-stochastic function $q: {\mathcal F}({\bf Z}) \rightarrow {\mathcal F}(B(\Gamma))$, we compute the fuzzy QoI $\tilde{Q}= Q(\tilde{\bf z}) = {\mathbb E}[q({\bf y}, \tilde{\bf z})]$ as follows.

\medskip
\STATE {\bf 1.} {\it Interaction}: For a fixed $\alpha \in [0,1]$, find the input joint $\alpha$-cut $S_{\alpha}^{\tilde{\bf z}}$ based on the interaction between the input fuzzy variables.

\medskip
\STATE {\bf 2.} {\it Optimization}: Obtain the lower and upper bounds of the output $\alpha$-cut $S_{\alpha}^{\tilde{Q}}$ 
by solving two global optimization problems:
\begingroup
\setlength\abovedisplayskip{5pt}
\setlength\belowdisplayskip{5pt}
$$
\ubar{Q} : =  \min_{{\bf z} \in S_{\alpha}^{\tilde{\bf z}}}  Q({\bf z}), \qquad 
\bar{Q} : = \max_{{\bf z} \in S_{\alpha}^{\tilde{\bf z}}} Q({\bf z}), \qquad 
Q({\bf z}) = {\mathbb E}[q({\bf y}, {\bf z})].
$$
\endgroup
%
%
%
%

\medskip
\STATE {\bf 3.} Repeat steps {\bf 1}-{\bf 2} for various levels of $\alpha \in [0,1]$. 

\end{algorithmic}
\end{algorithm}

An iterative optimization method fo solving the optimization problems in step 2 requires $M_f$ function evaluations $Q({\bf z}^{(k)})$ at $M_f$ fixed points $\{ {\bf z}^{(k)} \}_{k=1}^{M_f} \in S_{\alpha}^{\tilde{\bf z}}$. Each function evaluation amounts to computing the expectation ${\mathbb E}[q({\bf y},{\bf z}^{(k)})]$, which may be done by a Monte Carlo sampling strategy \cite{MC,Giles:11,QMC:11,Multiindex:16,MOMC:18} or a spectral stochastic method \cite{BTZ:05,Xiu_Hesthaven,Motamed_etal:13}, depending on the regularity of $q$ with respect to ${\bf y}$.

\subsection{Fuzzy-stochastic fields}
\label{sec:FS3}

Since the solution of fuzzy-stochastic PDEs are functions of space/time in addition to being functions of random and fuzzy vectors, the notion of fuzzy-stochastic functions needs to be extended to include the dependency on space/time. A scalar fuzzy-stochastic field is indeed a particular type of a fuzzy Banach space-valued function with a vector of spatial variables, a random vector, and a fuzzy vector as arguments. Similarly, one can include a temporal variable as argument and define fuzzy-stochastic processes.
\begin{definition}\label{fuzzy_stoch_field_def} 
Let $D \subset {\mathbb R}^d$ be a compact spatial domain, with $d=1,2,3$, and consider a vector of spatial variables ${\bf x} \in D$. Let further ${\bf y} \in \Gamma \subset {\mathbb R}^m$ and $\tilde{\bf z} \in {\mathcal F}({\bf Z})$ be a random vector with a bounded joint PDF and a fuzzy vector on ${\bf Z} \subset {\mathbb R}^n$, respectively. 
A scalar fuzzy-stochastic field, written as $\tilde u ({\bf x},{\bf y}) = u({\bf x}, {\bf y}, \tilde{\bf z}), \, \forall {\bf x} \in D, \, \forall {\bf y} \in \Gamma$, is a fuzzy Banach space-valued function 
$u: {\mathcal F}({\bf Z}) \rightarrow  {\mathcal F}(B(D \times\Gamma))$, 
where $B(D \times \Gamma)$ is a Banach space of functions on $D \times \Gamma$. 
\end{definition}

An example of $B(D \times\Gamma)$ in the context of fuzzy-stochastic PDEs is the Hilbert space of functions formed by the tensor product of two Hilbert spaces $H^s(D) \otimes L_{\pi}^2(\Gamma)$.

%

\section{Fuzzy-Stochastic PDEs}
\label{sec:SFPDE}

In general, we refer to PDEs with fuzzy-stochastic parameters, including coefficients, force terms, and boundary/initial data, as {\it fuzzy-stochastic PDEs}. Without loss of generality, in the present work, we consider only the case where the PDE coefficient is a fuzzy-stochastic field and assume that the forcing and data functions are deterministic. 

\subsection{A fuzzy-stochastic elliptic model problem}
\label{sec:SF_Elliptic}

Let $D \subset {\mathbb R}^d$ be a bounded, convex, Lipschitz spatial domain, with $d=1,2,3$. Consider the following fuzzy-stochastic elliptic boundary value problem:
\begin{equation}\label{FSPDE}
\begin{array}{ll}
- \nabla_{\bf x} \cdot ( a({\bf x},{\bf y},\tilde{\bf z}) \ \nabla_{\bf x} u({\bf x},{\bf y},\tilde{\bf z}) )= f({\bf x}),  & \ \ \ \ ({\bf x},{\bf y}) \in  D \times \Gamma,  \ \ \ \ \, \tilde{\bf z} \in {\mathcal F}({\bf Z}), \\
u({\bf x},{\bf y},\tilde{\bf z}) = 0, &  \ \ \ \ ({\bf x},{\bf y}) \in \partial D \times \Gamma,  \ \ \  \tilde{\bf z} \in  {\mathcal F}({\bf Z}),
\end{array}    
\end{equation}
where ${\bf x} = (x_1, \dotsc, x_d) \in D$ is the vector of spatial variables, ${\bf y} = (y_1, \dotsc, y_m) \in \Gamma \subset {\mathbb R}^m$ is a random vector with a bounded joint PDF $\pi = \pi({\bf y}): \Gamma \rightarrow {\mathbb R}_+$, and $\tilde{\bf z} = (\tilde{z}_1, \dotsc, \tilde{z}_n)\in {\mathcal F}({\bf Z})$ is a fuzzy vector on ${\bf Z} \subset {\mathbb R}^n$ satisfying assumptions (A4)-(A6), i.e. $\tilde{\bf z} \in {\mathcal F}_c({\bf Z})$, and with a family of joint $\alpha$-cuts $S_{\alpha}^{\tilde{\bf z}} \subset {\bf Z}$ with $\alpha \in [0,1]$. The only source of uncertainty is assumed to be the parameter $a$ characterized by a fuzzy-stochastic field $\tilde a ({\bf x},{\bf y}) = a({\bf x},{\bf y},\tilde{\bf z})$. This implies that the PDE solution $\tilde u ({\bf x},{\bf y}) = u({\bf x},{\bf y},\tilde{\bf z})$ is a fuzzy-stochastic field; see Section \ref{sec:solution} for the definition of the PDE solution. 

We assume that $m$ and $n$ are finite numbers. 
We further assume 
\begin{equation}\label{Assum1_f}
f \in L^2(D),
\vspace{-.1cm}
\end{equation}
\begin{equation}\label{Assum2_a}
0 < a_{\text{min}} \le \tilde a({\bf x},{\bf y}) \le a_{\text{max}} < \infty, \qquad \forall \, {\bf x} \in D, \ \ \forall \,  {\bf y} \in \Gamma.
\vspace{-.1cm}
\end{equation}
Assumption \eqref{Assum1_f} states that the forcing function $f$ is square integrable, and assumption \eqref{Assum2_a} states that the PDE coefficient at every fixed $({\bf x}, {\bf y}) \in D \times \Gamma$ is a uniformly positive and bounded fuzzy variable in the sense of Definition \ref{greater_smaller_fuzzy}. 

\subsection{Solution of the fuzzy-stochastic problem}
\label{sec:solution}

Following Definition \ref{Sobolev_fuzzy_fcn_def} and Definition \ref{fuzzy_stoch_field_def}, 
we interpret the solution $\tilde u ({\bf x},{\bf y}) = u({\bf x},{\bf y},\tilde{\bf z})$ to \eqref{FSPDE}, under assumptions \eqref{Assum1_f}-\eqref{Assum2_a}, as a fuzzy Banach space-valued function:
\begin{equation}\label{mapping1}
u: {\mathcal F}({\bf Z}) \rightarrow  {\mathcal F}(B(D \times \Gamma)), \qquad \text{where} \ \ \ \ B(D \times \Gamma) = H_0^1(D) \otimes L_{\pi}^2(\Gamma).
\end{equation}
Here, the function space $B(D \times \Gamma)$ is formed by the tensor product of two Hilbert spaces: $H_0^1(D)$ is the closure of the space of smooth functions with compact support in the Hilbert space of functions whose first weak derivatives are square integrable; and $L_{\pi}^2(\Gamma)$ is the Hilbert space of random functions with bounded second moments. 

For the convenience of both analysis and computation, and thanks to the {\it function-set identity} \eqref{function_set1}, which will be shown to hold (see Theorem \ref{wellposed}) due to the continuity of the mapping in \eqref{mapping1}, we will define the solution to \eqref{FSPDE} through the  corresponding parametric problem:
\begin{equation}\label{PFSPDE}
\begin{array}{lll}
- \nabla_{\bf x} \cdot ( a({\bf x},{\bf y},{\bf z}) \ \nabla_{\bf x} u({\bf x},{\bf y},{\bf z}) )= f({\bf x}),  & \text{in  }  & D \times \Gamma \times S_{\alpha}^{\tilde{\bf z}}, \\
u({\bf x},{\bf y},{\bf z}) = {\bf 0}, &  \text{on } & \partial D \times \Gamma \times S_{\alpha}^{\tilde{\bf z}},
\end{array}    
\end{equation}
where, following the assumption \eqref{Assum2_a}, 
\begin{equation}\label{Assum2_a2}
 0<a_{min} \le a({\bf x},{\bf y},{\bf z})  \le a_{max} < \infty, \qquad \forall \, {\bf x}\in D, \ \   \forall \, {\bf y}\in \Gamma, \ \ \forall \, {\bf z} \in S_{\alpha}^{\tilde{\bf z}}.
 \end{equation}
Corresponding to the interpretation \eqref{mapping1}, we interpret the solution $u({\bf x},{\bf y},{\bf z})$ to the parametric problem \eqref{PFSPDE} as a Banach space-valued function on $S_{\alpha}^{\tilde{\bf z}}$:
\begin{equation}\label{mapping2}
u: S_{\alpha}^{\tilde{\bf z}} \rightarrow B(D \times \Gamma), \qquad \text{where} \ \ \ \ B(D \times \Gamma) = H_0^1(D) \otimes L_{\pi}^2(\Gamma).
\end{equation}
As we will show later (see Theorem \ref{wellposed}), the mapping in \eqref{mapping2} is continuous and $u$ is uniformly bounded on $S_{\alpha}^{\tilde{\bf z}}$, i.e. $u \in L^{\infty}(S_{\alpha}^{\tilde{\bf z}}; \, H_0^1(D) \otimes L_{\pi}^2(\Gamma))$. 
This suggests that we may obtain the $\alpha$-cuts of the solution \eqref{mapping1} to the fuzzy-stochastic problem \eqref{FSPDE} from the extrema of the solution \eqref{mapping2} to the parametric problem \eqref{PFSPDE} on $S_{\alpha}^{\tilde{\bf z}}$:
\begin{equation}\label{new_star}
S_{\alpha}^{\tilde{u}} ({\bf x},{\bf y})= [\min_{{\bf z} \in S_{\alpha}^{\tilde{\bf z}}} {u}({\bf x},{\bf y},{\bf z}),
\max_{{\bf z} \in S_{\alpha}^{\tilde{\bf z}}} {u}({\bf x},{\bf y},{\bf z})] =: [\ubar{u}_{\alpha}({\bf x},{\bf y}), \bar{u}_{\alpha}({\bf x},{\bf y})], \ \ \alpha \in [0,1].
\end{equation}
Note that the lower and upper limits of the $\alpha$-cuts in \eqref{new_star} are stochastic fields. We also notice that the solution $u$ to \eqref{PFSPDE} is $\alpha$-dependent. However, for ease of notation, we omit the explicit dependence on $\alpha$ when no ambiguity arises.




\subsection{Well-posedness and regularity analysis}

The interpretation of the solution to fuzzy-stochastic PDE problems through Banach space-valued functions, i.e. the mappings \eqref{mapping1} and \eqref{mapping2}, simplifies the analysis of such problems. 
It transforms the original problem into a parametric one, as done in the case of pure stochastic PDEs. We can therefore extend the proofs for well-posedness and regularity of deterministic (see e.g. \cite{Evans}) and stochastic (see e.g. \cite{BNT:07,Motamed_etal:13,Motamed_etal:15}) PDE problems to fuzzy-stochastic PDE problems. 

We will hence base the analysis on the parametric representation \eqref{PFSPDE} of problem \eqref{FSPDE} and consider the following weak formulation of problem \eqref{PFSPDE} pointwise in ${\bf z} \in S_{\alpha}^{\tilde{\bf z}}$.

\medskip
\noindent
{\bf Weak formulation I.} Find $u: S_{\alpha}^{\tilde{\bf z}} \rightarrow H_0^1(D) \otimes L_{\pi}^2(\Gamma)$ such that $\forall {\bf z} \in S_{\alpha}^{\tilde{\bf z}}$ and for all test functions $v \in H_0^1(D) \otimes L_{\pi}^2(\Gamma)$ the following holds:
\begin{equation}\label{weak_formula}
\int_{D \times \Gamma}   \hskip -.3cm a({\bf x},{\bf y},{\bf z})  \nabla_{\bf x} u({\bf x},{\bf y},{\bf z}) \cdot \nabla_{\bf x} v({\bf x},{\bf y})  \pi({\bf y}) d{\bf y} d{\bf x}
= \int_{D \times \Gamma}   \hskip -.1cm   f({\bf x})  v({\bf x},{\bf y})  \pi({\bf y})  d{\bf y}  d{\bf x}.
\end{equation}
Such a solution, provided it exists, is referred to as a weak solution to \eqref{PFSPDE}. 


\begin{theorem}\label{wellposed}
Under the assumptions \eqref{Assum1_f} and \eqref{Assum2_a2}, there exists a unique weak solution $u \in C^0(S_{\alpha}^{\tilde{\bf z}}; \, H_0^1(D) \otimes L_{\pi}^2(\Gamma))$ to the parametric problem \eqref{PFSPDE}. Moreover, the solution depends continuously on the data.
\end{theorem}
\begin{proof} 
The proof is an easy extension of the proof for deterministic problems; see e.g. Section 6 of \cite{Evans}.
\end{proof}

By Theorem \ref{fuzzy_set_thm} and the continuity of the mapping $u: S_{\alpha}^{\tilde{\bf z}} \rightarrow H_0^1(D) \otimes L_{\pi}^2(\Gamma)$ by Theorem \ref{wellposed}, we have the function-set identity 
\begin{equation}\label{set_function_sec33}
S_{\alpha}^{\tilde{u}} ({\bf x},{\bf y}) = {u}({\bf x},{\bf y},S_{\alpha}^{\tilde{\bf z}}), \qquad \forall {\bf x} \in D, \quad \forall {\bf y} \in \Gamma.
\end{equation}
The lower and upper limits of the $\alpha$-cuts of solution \eqref{mapping1} to the fuzzy-stochastic problem \eqref{FSPDE} may then be obtained from the extrema of the solution \eqref{mapping2} to the parametric problem \eqref{PFSPDE}. In particular, provided the solution $u=u({\bf x}, {\bf y}, {\bf z})$ is a continuous function for every fixed point $({\bf x}, {\bf y}) \in D \times \Gamma$, its $\alpha$-cuts $S_{\alpha}^{\tilde{u}}({\bf x}, {\bf y})$ given by \eqref{set_function_sec33} will be compact, nested intervals given by \eqref{new_star} 
and satisfying $S_{\alpha_2}^{\tilde{u}} \subset S_{\alpha_1}^{\tilde{u}}$ with $0 \le \alpha_1 \le \alpha_2 \le 1$. We notice that in the absence of pointwise continuity, an $\alpha$-cut in \eqref{set_function_sec33} may be the union of disjoint intervals, and \eqref{new_star} may not hold. Nevertheless, in this latter case, we can still use \eqref{set_function_sec33} and obtain the lower and upper limits of $S_{\alpha}^{\tilde{u}} ({\bf x},{\bf y})$. 


As a corollary of Theorem \ref{wellposed}, we have the following result.
\begin{corollary}\label{wellposed2}
Consider the fuzzy-stochastic PDE problem \eqref{FSPDE} under the assumptions \eqref{Assum1_f}-\eqref{Assum2_a}. 
There exists a unique solution $\tilde u \in {\mathcal F}(H_0^1(D) \otimes L_{\pi}^2(\Gamma))$ that depends continuously on the data.
\end{corollary}


The inclusion property of the $\alpha$-cuts of the solution allows for efficient computations in fuzzy space for two reasons. First, it will allow us to restrict fuzzy computations to only a few $\alpha \in [0,1]$ levels, for example $\alpha=0, 0.25, 0.5, 0.75, 1$. After computing $S_{\alpha}^{\tilde{u}}$ for these $\alpha$ values, the output membership function can be constructed by interpolation. Secondly, since the zero-cut $S_0^{\tilde{\bf z}}$ contains all other $\alpha$-cuts, i.e., $S_{\alpha}^{\tilde{\bf z}} \subset S_0^{\tilde{\bf z}}, \ \forall \alpha \in (0,1]$, we will need to construct the response surface of the solution $u({\bf x}, {\bf y}, .)$ only over the zero-cut. Hence we solve the parametric problem \eqref{PFSPDE} over the zero-cut. The response surface of the solution over any desired $\alpha$-cut may then be obtained by restricting the zero-cut response surface to the desired $\alpha$-cut.



To study the regularity of the solution to the parametric problem \eqref{PFSPDE} with respect to parameters both in stochastic space and in fuzzy space, we combine the stochastic and fuzzy spaces and let $\boldsymbol\xi = ({\bf y}, {\bf z})$ be the parameter vector in the {\it combined} stochastic-fuzzy space 
$\Xi := \Gamma \times S_0^{\tilde{\bf z}} \subset {\mathbb R}^N$, where $N = m + n$. 
We then study the ${\boldsymbol\xi}$-regularity of the solution $u({\bf x},\boldsymbol\xi)$ to \eqref{PFSPDE}. 
We note that since the solution over any desired $\alpha$-cut $S_{\alpha}^{\tilde{\bf z}}$, with $\alpha \in [0,1]$, can be obtained by restricting the zero-cut solution to the desired $\alpha$-cut, we need to consider only the zero-cut $S_0^{\tilde{\bf z}}$ in $\Xi$. Indeed, the regularity of the solution over the zero-cut will determine the regularity of the solution over all other $\alpha$-cuts.

We view the solution to the parametric problem \eqref{PFSPDE} as a function of $\boldsymbol\xi \in \Xi$ taking values in the Hilbert space $H_0^1(D)$ and study the regularity of the mapping $u:  \Xi \rightarrow H_0^1(D)$. 
In the light of this interpretation we consider the following weak formulation of problem \eqref{PFSPDE} pointwise in $\boldsymbol\xi \in \Xi$.

\medskip
\noindent
{\bf Weak formulation II.} Find $u: \Xi \rightarrow H_0^1(D)$ such that $\forall \boldsymbol\xi \in \Xi$ and for all test functions $v \in H_0^1(D)$ the following holds:
\begin{equation}\label{weak_formula_xi}
B[u,v] = f(v),
\vspace{-.2cm}
\end{equation}
\vspace{-.5cm}
\begin{equation}\label{B_f}
B[u,v] := \int_{D} a({\bf x},\boldsymbol\xi) \ \nabla_{\bf x} u({\bf x},\boldsymbol\xi) \cdot \nabla_{\bf x} v({\bf x}) \, d{\bf x}, \qquad f(v) = \int_{D} f({\bf x}) \, v({\bf x}) \, d{\bf x}.
\end{equation}
%
Clearly, following Theorem \ref{wellposed}, there exist a unique solution $u \in L^{\infty}(\Xi; H_0^1(D))$ to the parametric problem \eqref{PFSPDE}.

For regularity analysis we will also need some regularity assumptions on the $\boldsymbol\xi$-regularity of the PDE coefficient. 
Let ${\bf k} = (k_1, \dotsc, k_N) \in {\mathbb N}^N$ be a multi-index with $|{\bf k}| = k_1 + \dotsc + k_N$ and  ${\mathbb N}$ denoting the set of all non-negative integers including zero. We make the following regularity assumption on the PDE coefficient,
\vspace{-.1cm}
\begin{equation}\label{assum_reg}
\partial_{\boldsymbol\xi}^{|{\bf k}|} a( ., \boldsymbol\xi) 
: = \frac{\partial^{|{\bf k}|} a( ., \boldsymbol\xi)}{\partial_{\xi_1}^{k_1} \, \dotsc \partial_{\xi_N}^{k_N}} 
\in L^{\infty}(D), \quad
0 \le |{\bf k}| \le s, \ \ \ s \in {\mathbb N}, \qquad
\forall \, \boldsymbol\xi \in \Xi.
\vspace{-.1cm}
\end{equation}
We can now state the following regularity result.

\begin{theorem}\label{regularity2}
For the solution of the parametric problem \eqref{PFSPDE} with the forcing term satisfying \eqref{Assum1_f} and the coefficient satisfying \eqref{Assum2_a2} and \eqref{assum_reg}, we have,
$$
\partial_{\boldsymbol\xi}^{|{\bf k}|} {u}: = \frac{\partial^{|{\bf k}|} u}{\partial_{\xi_1}^{k_1} \, \dotsc \partial_{\xi_N}^{k_N}}  \in L^{\infty}(\Xi; H_0^1(D)), \qquad {\bf k} \in {\mathbb N}^N, \quad 0\le | {\bf k} | \le s.
$$
\end{theorem}
\begin{proof}
The proof is an easy extension of the proof for stochastic problems; see e.g. \cite{Motamed_etal:13,Motamed_etal:15}.
\end{proof}

\section{Numerical Experiments}
\label{sec:numerics}

In this section we present two numerical examples. 
In both examples, we consider the following fuzzy-stochastic elliptic problem:
\vspace{-.4cm}
\begin{subequations}\label{FSPDE_1D}
\begin{gather}
\frac{d}{dx} \bigl( a(x,{\bf y},\tilde{\bf z}) \, \frac{du}{dx}(x,{\bf y},\tilde{\bf z}) \bigr) = 0, \qquad x \in [0,L], \quad {\bf y} \in \Gamma, \quad \tilde{\bf z} \in {\mathcal F}({\bf Z}), \label{PDE_1D}\\
u(0,{\bf y},\tilde{\bf z})= 0, \quad a(L,{\bf y},\tilde{\bf z}) \,\frac{du}{dx} (L,{\bf y},\tilde{\bf z})= 1. \label{BCs}
\end{gather}
\end{subequations}
Here, the source of uncertainty is the parameter $a$ characterized by a fuzzy-stochastic field; see the two examples below. The solution to \eqref{FSPDE_1D} is analytically given by
\vspace{-.2cm}
\begin{equation}\label{FSPDE_1D_sol}
u(x, {\bf y}, \tilde{\bf z}) = \int_0^x a^{-1} (\xi, {\bf y} , \tilde{\bf z}) \, d\xi.
\vspace{-.2cm}
\end{equation}
We note that in more complex problems in higher dimensions, the PDE problem needs to be discretized on the spatial domain by a numerical method. Here, we consider the one-dimensional problem \eqref{FSPDE_1D} and focus on computations in fuzzy and stochastic spaces. 
We end this section by a short, general discussion on the computational cost of fuzzy-stochastic PDE problems.

\subsection{Numerical example 1}
\label{sec:numerics1}
As an illustrative example, we let $L=2$ and consider \eqref{FSPDE_1D} with the fuzzy-stochastic parameter
$$
a(x,y,\tilde{\bf z}) = a_1(x) \, a_2(y,\tilde{\bf z}) = (2 + \sin(2 \pi \, x/L)) \, e^{\tilde z_1 + y \, \tilde z_2}, \qquad y \sim {\mathcal N}(0,1), \quad \tilde{\bf z} = (\tilde z_1, \tilde z_2).
$$
Here, $a$ is a fuzzy-stochastic field, given by the product of a deterministic function $a_1(x)$ and a fuzzy-stochastic function $a_2(y,\tilde{\bf z})$, being a lognormal random variable with a fuzzy mean $\tilde z_1$ and a fuzzy standard deviation $\tilde z_2$, i.e. $
a_2(y,\tilde{\bf z}) \sim \ln \, {\mathcal N}[\tilde z_1, \tilde z_2^2]$. We assume that $\tilde{z}_1$ and $\tilde{z}_2$ are triangular numbers, that is, they have triangular-shaped membership functions, uniquely described by triples $\langle z_i^l,z_i^m, z_i^r \rangle$, where $z_i^l < z_i^m <z_i^r$ and such that $\mu_{\tilde{z}_i}(z_i^l) = \mu_{\tilde{z}_i}(z_i^r) = 0$ and $\mu_{\tilde{z}_i}(z_i^m)=1$, with $i=1,2$. 
The marginal $\alpha$-cuts of the two fuzzy variables are then given by
$$
S_{\alpha}^{\tilde{z}_1} = [z_{1}^l+\alpha \, (z_{1}^m-z_{1}^l), \, z_{1}^r-\alpha \, (z_{1}^r-z_{1}^m)] =: [a_{\alpha}, b_{\alpha}],
$$
$$
S_{\alpha}^{\tilde{z}_2} = [z_{2}^l+\alpha \, (z_{2}^m-z_{2}^l), \, z_{2}^r-\alpha \, (z_{2}^r-z_{2}^m)] =: [c_{\alpha}, d_{\alpha}].
$$
We consider two cases of non-interactive and fully interactive fuzzy variables. If $\tilde z_1$ and $\tilde z_2$ are non-interactive, following Definition \ref{def_noninteractive}, their joint $\alpha$-cut is 
$$
S_{\alpha}^{\tilde{\bf z}} = [a_{\alpha},b_{\alpha}]\times[c_{\alpha},d_{\alpha}].
$$
If $\tilde z_1$ and $\tilde z_2$ are fully interactive, 
following Definition \ref{def_fullyinteractive}, their joint $\alpha$-cut may be considered to be a piecewise linear curve in ${\mathbb R}^2$ with the Euclidean length 
$$
L_{\alpha} =  \sqrt{(a_{\alpha} - z_1^m)^2 + (c_{\alpha} - z_2^m)^2} + \sqrt{(b_{\alpha} - z_1^m)^2 + (d_{\alpha} - z_2^m)^2} =: L_{1,\alpha} + L_{2,\alpha}, 
$$
given by the collection of points
$$
S_{\alpha}^{\tilde{\bf z}} = \{ (z_1, z_2) = \boldsymbol\varphi(s), \, s \in [0, L_{\alpha}] \}, 
$$
with the piecewise linear map
$$
\boldsymbol\varphi(s) = \left\{ \begin{array}{l l}
(  a_{\alpha} + s |z_1^m - a_{\alpha}| / L_{1,\alpha}, \, c_{\alpha} + s |z_2^m - c_{\alpha}| / L_{1,\alpha} ) & \qquad s \in [0, L_{1,\alpha}]\\
(  z_1^m + s | b_{\alpha} - z_1^m | / L_{2,\alpha}, \, z_2^m + s | d_{\alpha} - z_2^m| / L_{2,\alpha} ) & \qquad s \in [L_{1,\alpha}, L_{\alpha}]
\end{array} \right..
$$

We consider the following QoIs
\begin{align*}
& \tilde{Q}_1 = Q_1(\tilde{\bf z}) = {\mathbb E}[u(L,y,\tilde{\bf z})], \\
& \tilde{Q}_2 (x) = Q_2(x,\tilde{\bf z}) = {\mathbb E}[u(x,y,\tilde{\bf z})], \\
& \tilde{Q}_3 (y) = Q_3(y,\tilde{\bf z}) = u(L,y,\tilde{\bf z}).
\end{align*}
These QoIs cover a wide range of fuzzy quantities: $\tilde{Q}_1$ is a fuzzy function; $\tilde{Q}_2$ is a fuzzy field; and $\tilde{Q}_3$ is a fuzzy-stochastic function. We now discuss the computation and visualization of each quantity in turn, based on Algorithm \ref{ALG_fuzzy_functions} and Algorithm \ref{ALG_fuzzy_random_functions}.

The computation of $\tilde{Q}_1$ amounts to computing its $\alpha$-cuts $S_{\alpha}^{\tilde{Q}_1}$ at various levels $\alpha \in [0,1]$. It requires evaluating the PDE solution \eqref{FSPDE_1D_sol} at a fixed point $x=L$ and for $M_s$ realizations $\{ y^{(i)} \}_{i=1}^{M_s}$ of $y\sim {\mathcal N}(0,1)$. 
The solution \eqref{FSPDE_1D_sol} at $x=L$ and a fixed realization $y^{(i)}$ is the integral of a fuzzy field $a^{-1} (\xi, y^{(i)}, \tilde{\bf z})$, with $\xi \in [0,L]$, over a crisp interval $[0,L]$. We approximate the crisp integral by a quadrature, such as the midpoint rule, and write
\vspace{-.3cm}
\begin{equation*}
u(L, y^{(i)},\tilde{\bf z}) = \int_0^L a^{-1}(\xi, y^{(i)}, \tilde{\bf z}) \, d\xi \approx h \, \sum_{j=1}^{N_h} a^{-1}(x_j, y^{(i)}, \tilde{\bf z}), \ \ \ x_j = (j-\frac1{2}) h, \ \ \ h= \frac{L}{N_h}.
\vspace{-.3cm}
\end{equation*}
Following Algorithm \ref{ALG_fuzzy_random_functions}, we employ the standard Monte Carlo sampling and write
\vspace{-.2cm}
\begin{equation*}
Q_1 (\tilde{\bf z}) = {\mathbb E}\biggl[ u(L, y^{(i)},\tilde{\bf z}) \biggr]  \approx 
\frac{1}{M_s} \, \sum_{i=1}^{M_s} u(L, y^{(i)},\tilde{\bf z}) \approx
\frac{h}{M_s} \, \sum_{i=1}^{M_s} \sum_{j=1}^{N_h} a^{-1}(x_j, y^{(i)}, \tilde{\bf z}).
\vspace{-.2cm}
\end{equation*}
Note that we can alternatively employ other Monte Carlo sampling strategies \cite{Giles:11,QMC:11,Multiindex:16,MOMC:18} or spectral stochastic techniques \cite{Xiu_Karniadakis:02,BTZ:05,Xiu_Hesthaven,Motamed_etal:13} to approximate the expectation. 
Finally, following Algorithm \ref{ALG_fuzzy_functions}, we perform the addition of $M_s N_h$ fuzzy functions $\{ a^{-1}(x_j, y^{(i)}, \tilde{\bf z}) \}$, with $j=1, \dotsc, N_h$ and $i = 1, \dotsc, M_s$, to get:
\vspace{-.2cm} 
\begin{equation}\label{S_Q1}
S_{\alpha}^{\tilde{Q}_1} = \Bigl[ \min_{{\bf z} \in S_{\alpha}^{\tilde{\bf z}}}  \bigl( \frac{h}{M_s} \, \sum_{i=1}^{M_s} \sum_{j=1}^{N_h} a^{-1}(x_j, y^{(i)}, {\bf z})  \bigr), \, \max_{{\bf z} \in S_{\alpha}^{\tilde{\bf z}}}  \bigl( \frac{h}{M_s} \, \sum_{i=1}^{M_s} \sum_{j=1}^{N_h} a^{-1}(x_j, y^{(i)}, {\bf z})  \bigr) \Bigr].
\vspace{-.2cm}
\end{equation}
We notice that since all $M_s N_h$ fuzzy functions $\{ a^{-1}(x_j, y^{(i)}, \tilde{\bf z}) \}$ are fully interactive, i.e. they are all functions of the same fuzzy vector $\tilde{\bf z}$, the $\alpha$-cuts are obtained by the extrema of the sum of the terms, rather than the sums of the extrema. The latter would give conservative intervals overestimating the true $\alpha$-cuts. 
After computing various $\alpha$-cuts \eqref{S_Q1} at different $\alpha$ levels we can construct the membership function of $\tilde{Q}_1$ by interpolation. 
Figure \ref{Ex_1D_Q1} shows the membership functions of $\tilde{Q}_1$ for two cases of interaction, and with
\begin{equation}\label{ex1_triangular_fuzzy}
\tilde z_1 =  \langle z_1^l, z_1^m, z_1^r  \rangle = \langle 1.00, 1.06,1.20  \rangle, \quad \tilde z_2 = \langle z_2^l, z_2^m, z_2^r  \rangle = \langle 0.10, 0.13,0.20  \rangle.
\end{equation}
We observe that $\mu_{\tilde{Q}_1}$ corresponding to non-interactive input fuzzy variables contains $\mu_{\tilde{Q}_1}$ corresponding to fully interactive fuzzy variables, as expected. 
\begin{figure}[!h]
\vspace{-.3cm}
\center
\includegraphics[width=6.4cm,height=3.7cm]{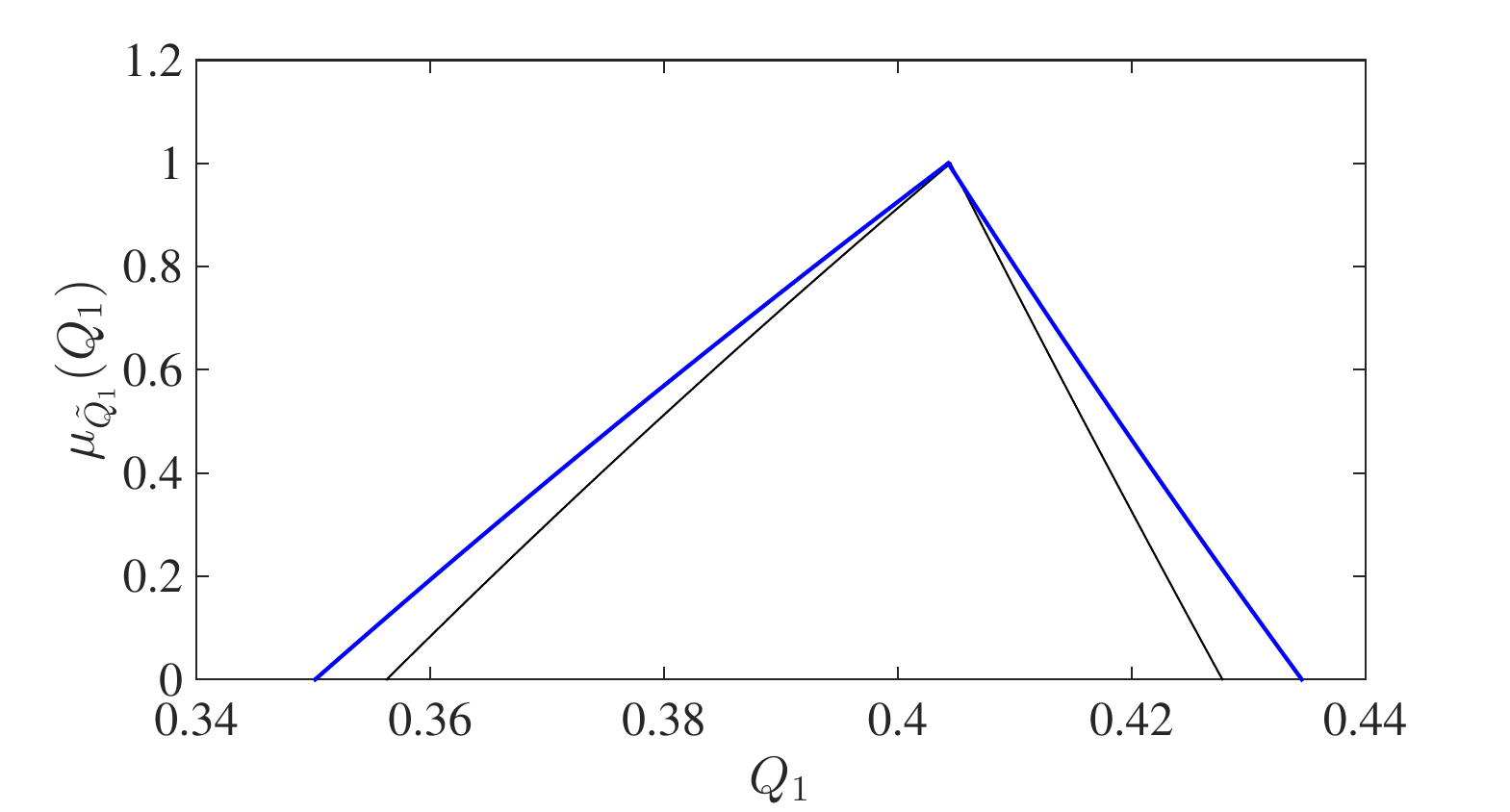}
\vspace{-.3cm}
    \caption{Membership functions of $\tilde{Q}_1$ when the input fuzzy variables are non-interactive (blue thick curves) and fully interactive (black think curves). The former contains the latter, as expected.}
    \label{Ex_1D_Q1}
 \end{figure}

\vspace{-.1cm}

The computation of the fuzzy field $\tilde{Q}_2$ is similar to that of $\tilde{Q}_1$ for different $x$ values. Figure \ref{Ex_1D_Q2} shows the fuzzy field $\tilde{Q}_2$ versus $x \in [1.8,2]$ with gray-scale colors representing the membership degrees, ranging from 0 (white color) to 1 (black color), in both non-interactive (left) and fully interactive (right) cases. We use the same values of parameters as those in \eqref{ex1_triangular_fuzzy}. While both cases result in similar fuzzy fields, the field obtained by non-interactive fuzzy variables does contain the field obtained by fully interactive fuzzy variables, as expected.
\begin{figure}[!h]
\vspace{-.5cm}
\center
\subfigure[Non-interactive fuzzy variables]{\includegraphics[width=5.5cm,height=4cm]{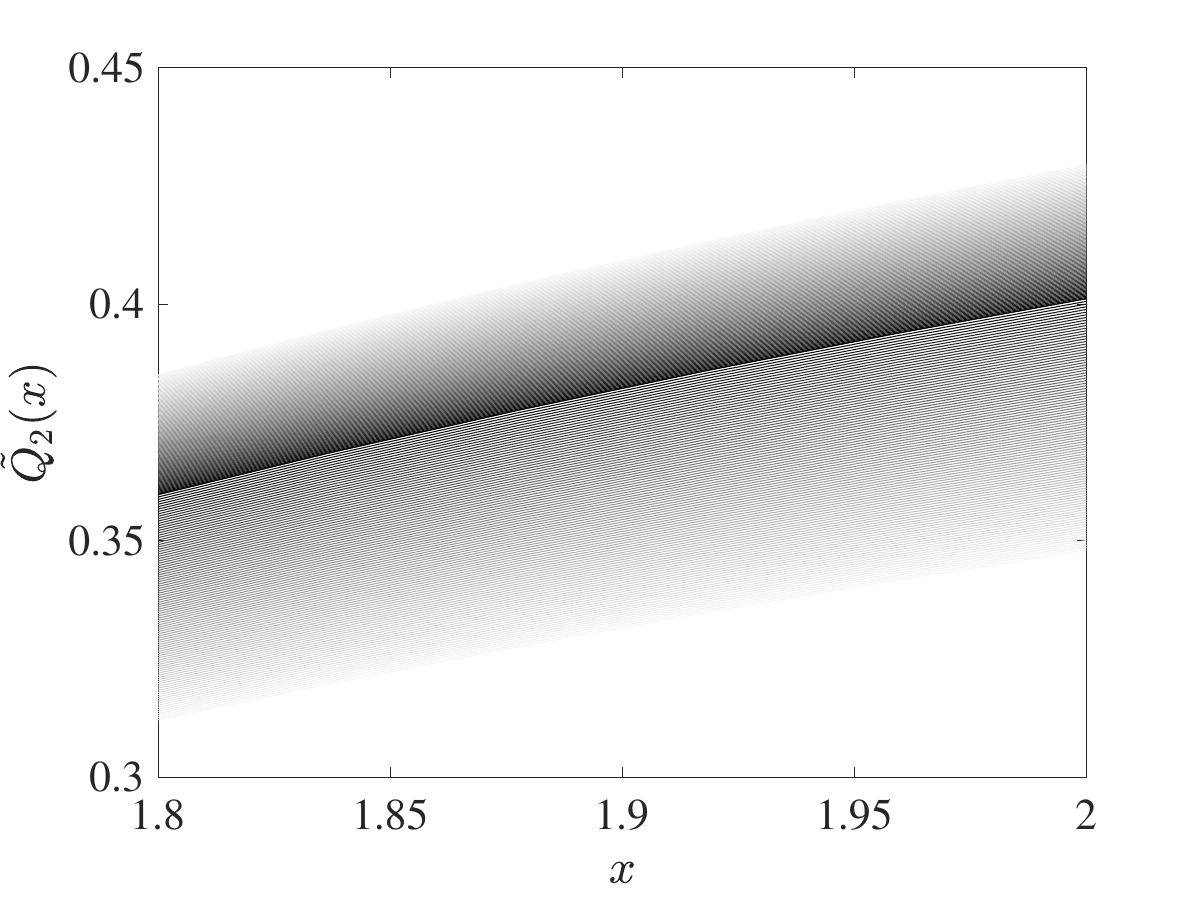}}
 \hskip .1cm
\subfigure[Fully interactive fuzzy variables]{\includegraphics[width=5.5cm,height=4cm]{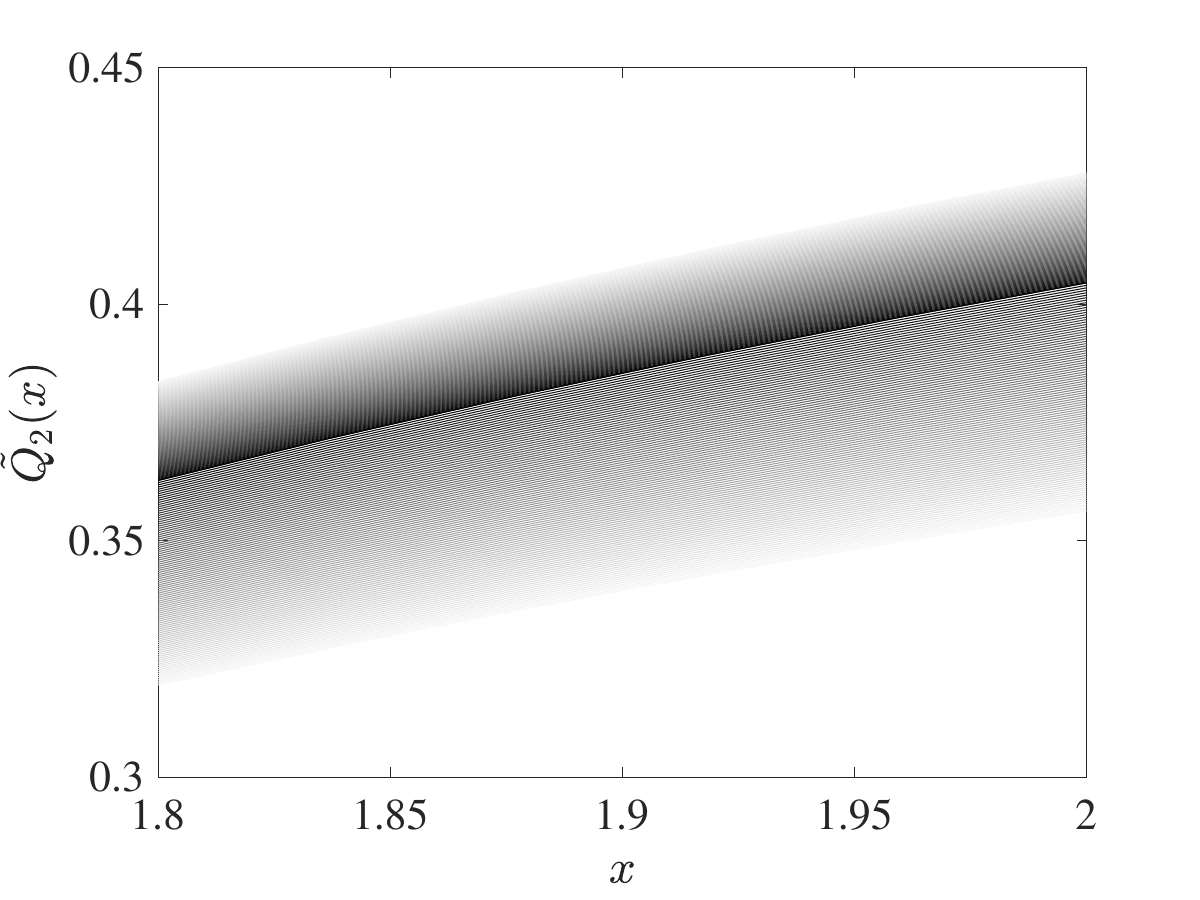}}\vspace{-.4cm}
    \caption{The fuzzy field $\tilde{Q}_2(x)$ versus $x$ with gray-scale colors representing the membership degrees, ranging from 0 (white color) to 1 (black color), with non-interactive (left) and fully interactive (right) input fuzzy variables. The fuzzy field in (a) contains the fuzzy field in (b).}
    \label{Ex_1D_Q2}
 \end{figure}

The computation of the fuzzy-stochastic function $\tilde{Q}_3(y)$ amounts to computing its fuzzy CDF. We follow the approach outlined in Algorithm \ref{ALG_fuzzy_random_functions}. For each fixed $Q_{3,0} \in [0.2,0.6]$, we compute the $\alpha$-cuts $S_{\alpha}^{\tilde{F}}$ of $\tilde{F}(Q_{3, 0}) = {\mathbb E}[{\mathbb I}_{[Q_3(y,\tilde{\bf z}) \le Q_{3,0}]}]$ and then construct the fuzzy CDF of $\tilde{Q}_3$. This corresponds to $q(y,\tilde{\bf z}) ={\mathbb I}_{[Q_3(y,\tilde{\bf z}) \le Q_{3,0}]}$ in Algorithm \ref{ALG_fuzzy_random_functions}. We use the parameter values in \eqref{ex1_triangular_fuzzy}. 
Figure \ref{Ex_1D_Q3} shows the fuzzy CDF of $\tilde{Q}_3$ with gray-scale colors representing the membership degrees, ranging from 0 (white color) to 1 (black color), in both non-interactive (left) and fully interactive (right) cases. 
To each membership degree (or $\alpha$-level), there corresponds one lower and one upper envelope, forming a p-box. We observe a nested set of p-boxes at different $\alpha$-levels: p-boxes at upper levels of plausibility/possibility (higher $\alpha$-levels) are contained inside p-boxes at lower levels of plausibility/possibility (lower $\alpha$-levels). 
Again, we observe that $\tilde{F}({Q}_3)$ corresponding to non-interactive input fuzzy variables contains $\tilde{F}({Q}_3)$ corresponding to fully interactive fuzzy variables. 
\begin{figure}[!h]
\vspace{-.57cm}
\center
\subfigure[Non-interactive fuzzy variables]{\includegraphics[width=6.4cm,height=3.7cm]{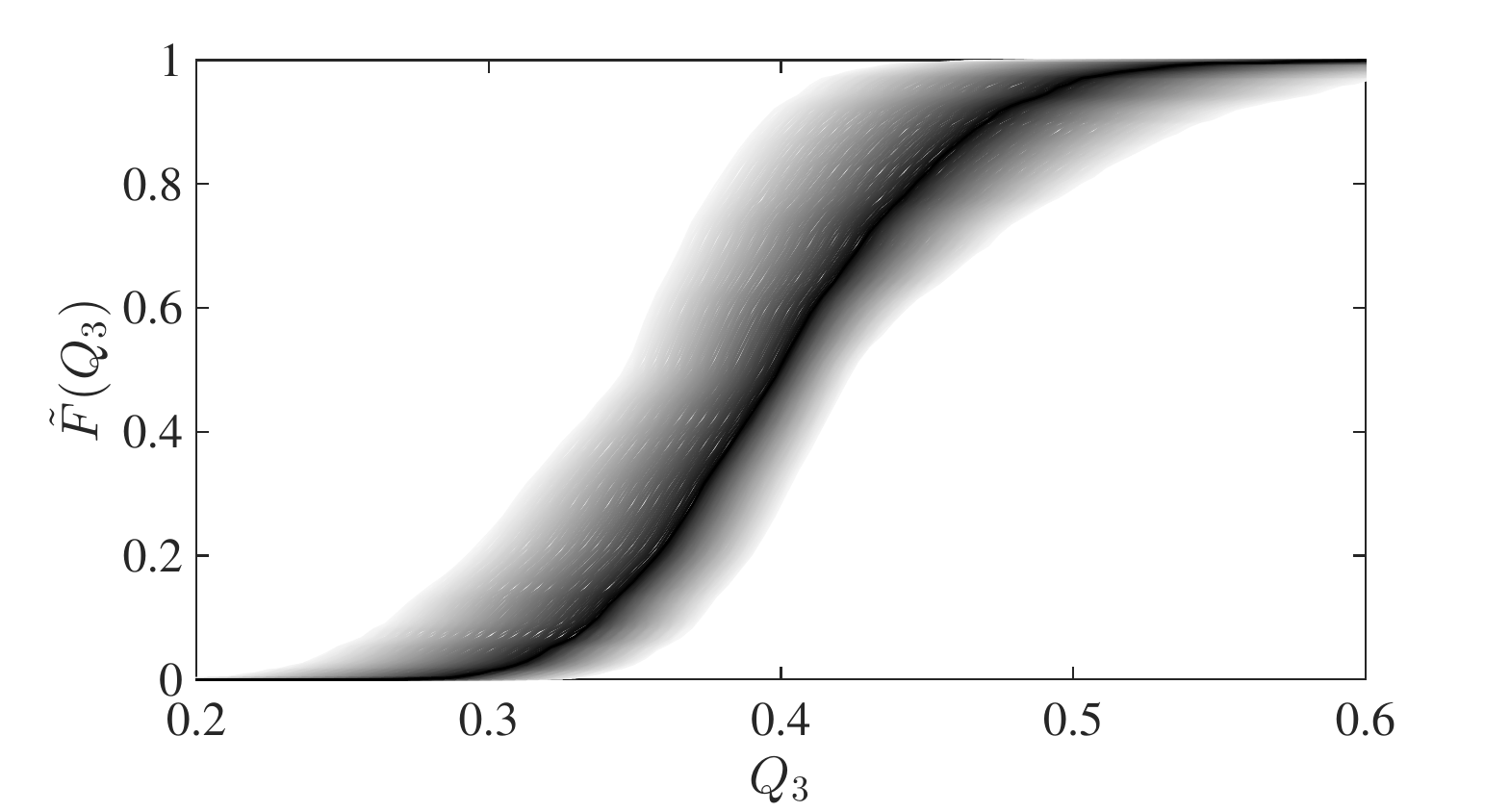}}
 \hskip .1cm
\subfigure[Fully interactive fuzzy variables]{\includegraphics[width=6.4cm,height=3.7cm]{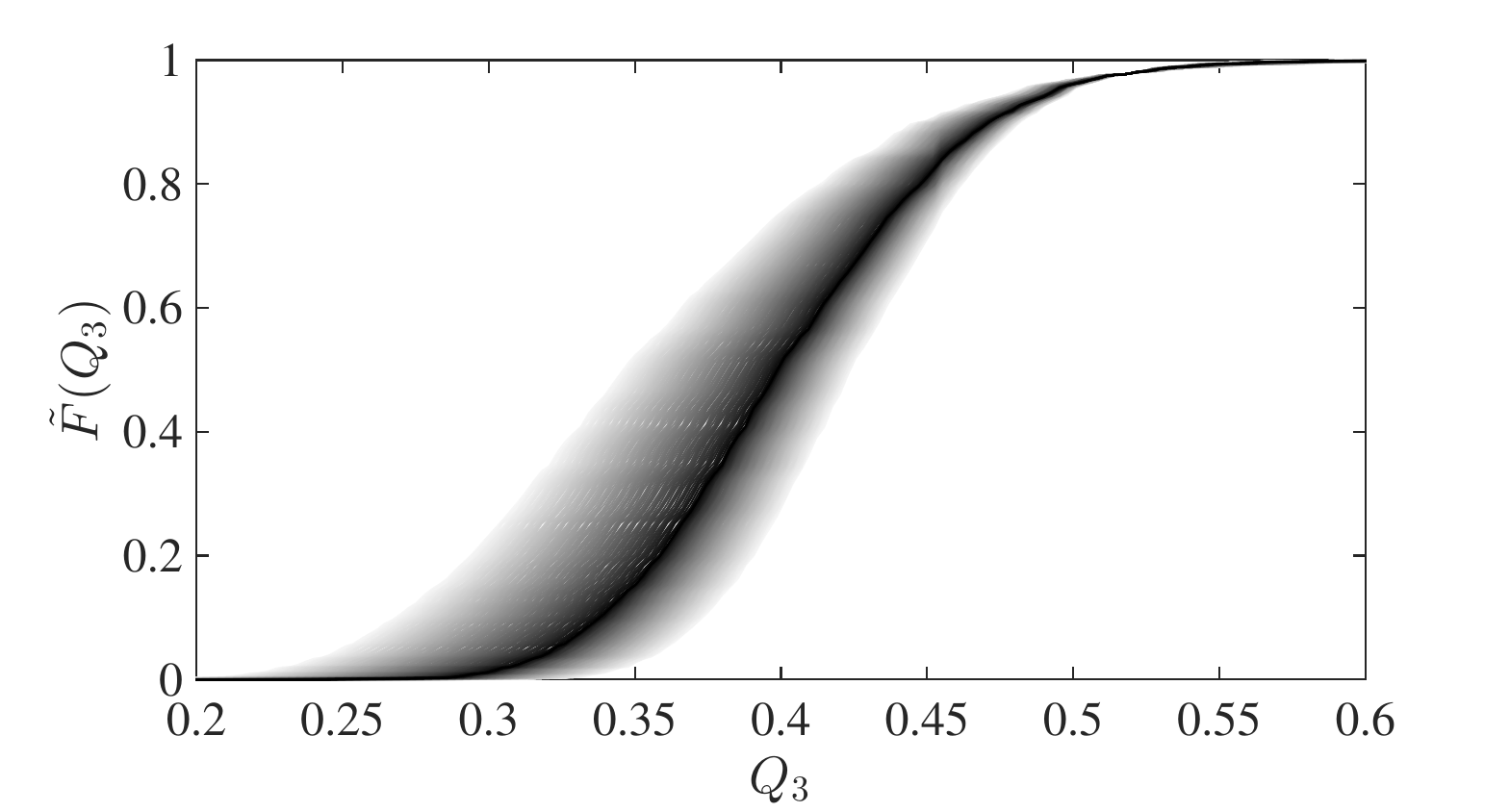}}
\vspace{-.4cm}
    \caption{The fuzzy CDF of $\tilde{Q}_3$ with gray-scale colors representing the membership degrees, ranging from 0 (white color) to 1 (black color), with non-interactive (left) and fully interactive (right) input fuzzy variables. To each membership degree, there corresponds one lower and one upper envelope, forming a p-box. A fuzzy CDF may then be viewed as a nested set of p-boxes at different levels of plausibility/possibility. The fuzzy CDF in (a) contains the fuzzy CDF in (b).}
    \label{Ex_1D_Q3}
 \end{figure}


\subsection{Numerical example 2}
\label{sec:numerics2}


We next consider an engineering problem in materials science: the response of fiber-reinforced polymers to external forces. This example demonstrates the applicability of fuzzy-stochastic PDEs to real-world problems. We will in particular show how fuzzy-stochastic PDE parameters can be constructed based on real measurement data. The construction will be justified and validated by showing that the PDE-generated outputs accurately capture variations in the true quantities obtained by real data. To this end, we consider a small piece of HTA/6376 fiber composite \cite{Babuska_etal:1999,BMT:14,BM:16}, consisting of four plies containing $13688$ carbon fibers with a volume fraction of $63 \%$ in epoxy matrix. Figure \ref{fiber_map}(top) shows a map of fibers in an orthogonal cross section of the composite obtained by an optical microscope. The modulus of elasticity of the composite constituents, given by the manufacturer, are $a_{\text{fiber}} = 24$ [GPa] and $a_{\text{matrix}} = 3.6$ [GPa]. We process this binary map and convert it into a form suitable for statistical analysis, based on which the modulus of elasticity of the composite $a$, which appears in \eqref{FSPDE_1D}, will be characterized. We follow \cite{BM:16} and discretize the rectangular cross section of the composite into a uniform mesh of square pixels of size $1 \times 1 \, \mu {\text{m}}^2$. We then construct a binary data structure for the composite's modulus of elasticity, where we mark the presence or absence of fiber at every pixel by 1 ($a = a_{\text{fiber}}$) or 0 ($a = a_{\text{matrix}}$), respectively, assuming that fibers are perfectly circular. We next divide the rectangular domain into $50$ thin horizontal strips (or bars) of width $10 \, \mu$m. This gives us $50$ thin bars of length $1700 \, \mu$m, labeled $i=1,\dotsc,50$. Each bar is divided into 170 square elements of size $10 \times 10 \, \mu {\text{m}}^2$, labeled $j=1, \dotsc, 170$. On each element $j$, we take the harmonic average over its $10 \times 10$ pixels and compute a value $a_i(x_j)$ for modulus of elasticity. 
We repeat the process for all $50$ bars and all $170$ elements of each bar and obtain $50$ one-dimensional discrete samples $\{ a_{i}(x_j) \}_{i=1}^{50}$ of the uncertain parameter $a$ at the discrete points $\{ x_j \}_{j=1}^{170}$; see Figure \ref{fiber_map}(bottom). The above process is accurate within 1$\%$ in predicting the overall volume fraction obtained by an analytic approach. We refer to \cite{BM:16} for details.
%
%
%
%
%
\begin{figure}[!h]
\vspace{-.2cm}
\center
\subfigure{\includegraphics[width=7cm,height=2.2cm]{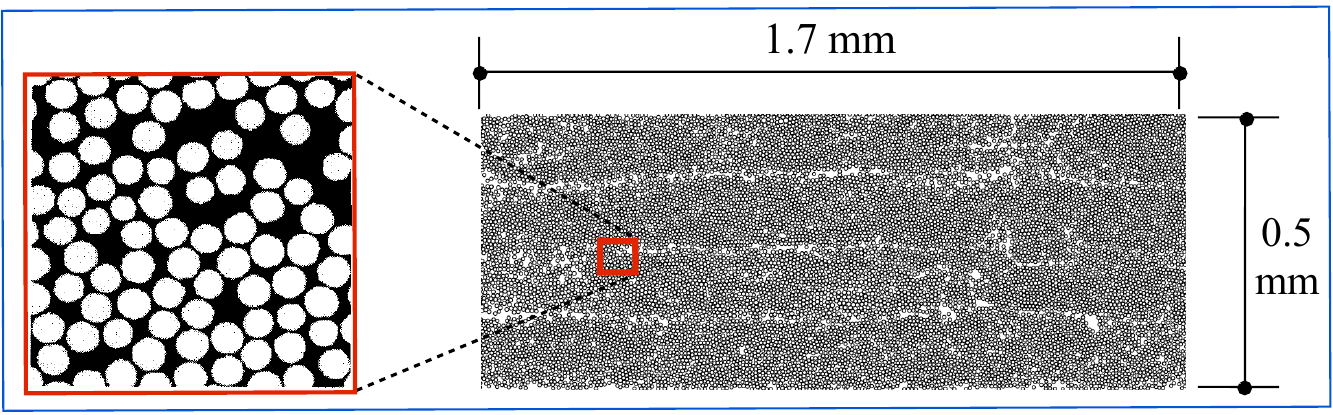}}
\\
\vspace{-.2cm}
\subfigure{\includegraphics[width=7cm,height=2.0cm]{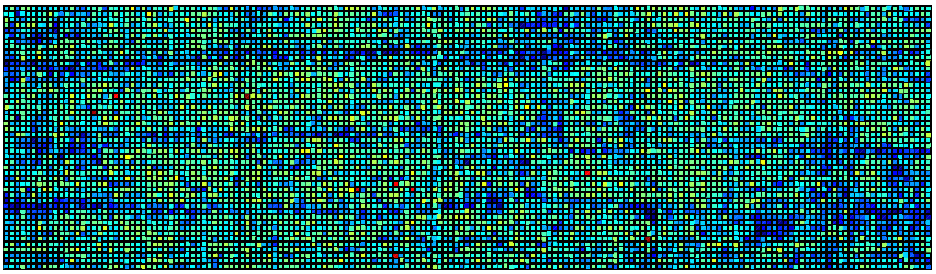}}
\vspace{-.3cm}
    \caption{Top: a binary optical image of a small piece of a fiber composite. Bottom: modulus of elasticity of composite over a regular mesh of $170 \times 50$ square elements of size $10 \times 10 \, \mu {\text{m}}^2$.}
\label{fiber_map}
\end{figure}


Motivated by the form of the exact solution \eqref{FSPDE_1D_sol}, we perform statistical analysis directly on the compliance $b = a^{-1}$. We approximate the first four moments of $b$ by sample averaging using the samples $\{ b_{i}(x_j) \}_{i=1}^{50} = \{ a_i^{-1}(x_j) \}_{i=1}^{50}$. At each discrete point $x_1, \dotsc, x_{170}$, we use these $50$ samples and compute the sample mean $z_1(x_j)$, sample standard deviation $z_2(x_j)$, sample skewness $z_3(x_j)$, and sample excess kurtosis $z_4(x_j)$. Figure \ref{sample_moments} shows the sample moments of $b$ versus $x$ and their histograms. 
\begin{figure}[!h]
\vskip -.55cm
  \begin{center}
        \subfigure[sample mean]{\includegraphics[width=6.2cm,height=4.3cm]{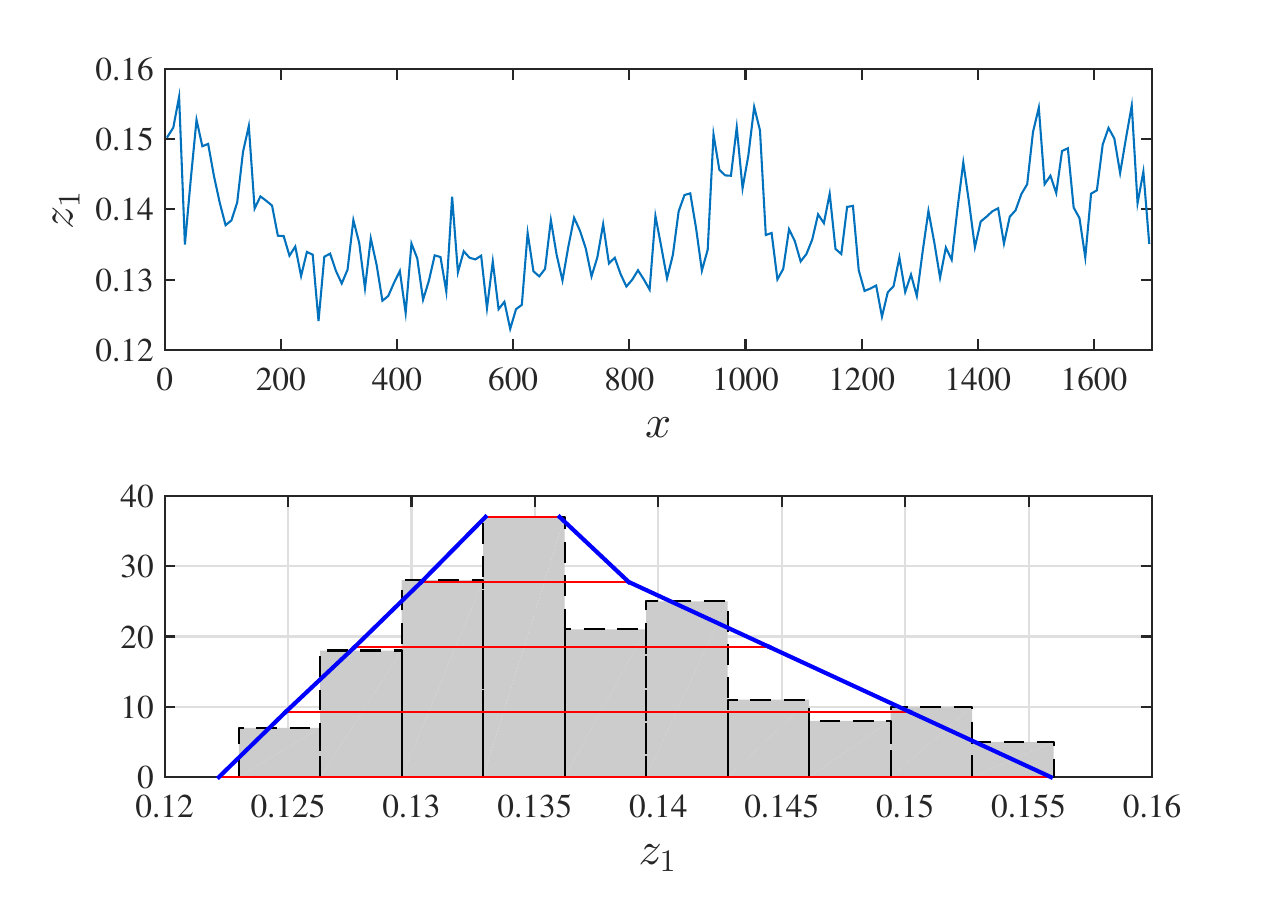}}
        \subfigure[sample standard deviation]{\includegraphics[width=6.2cm,height=4.3cm]{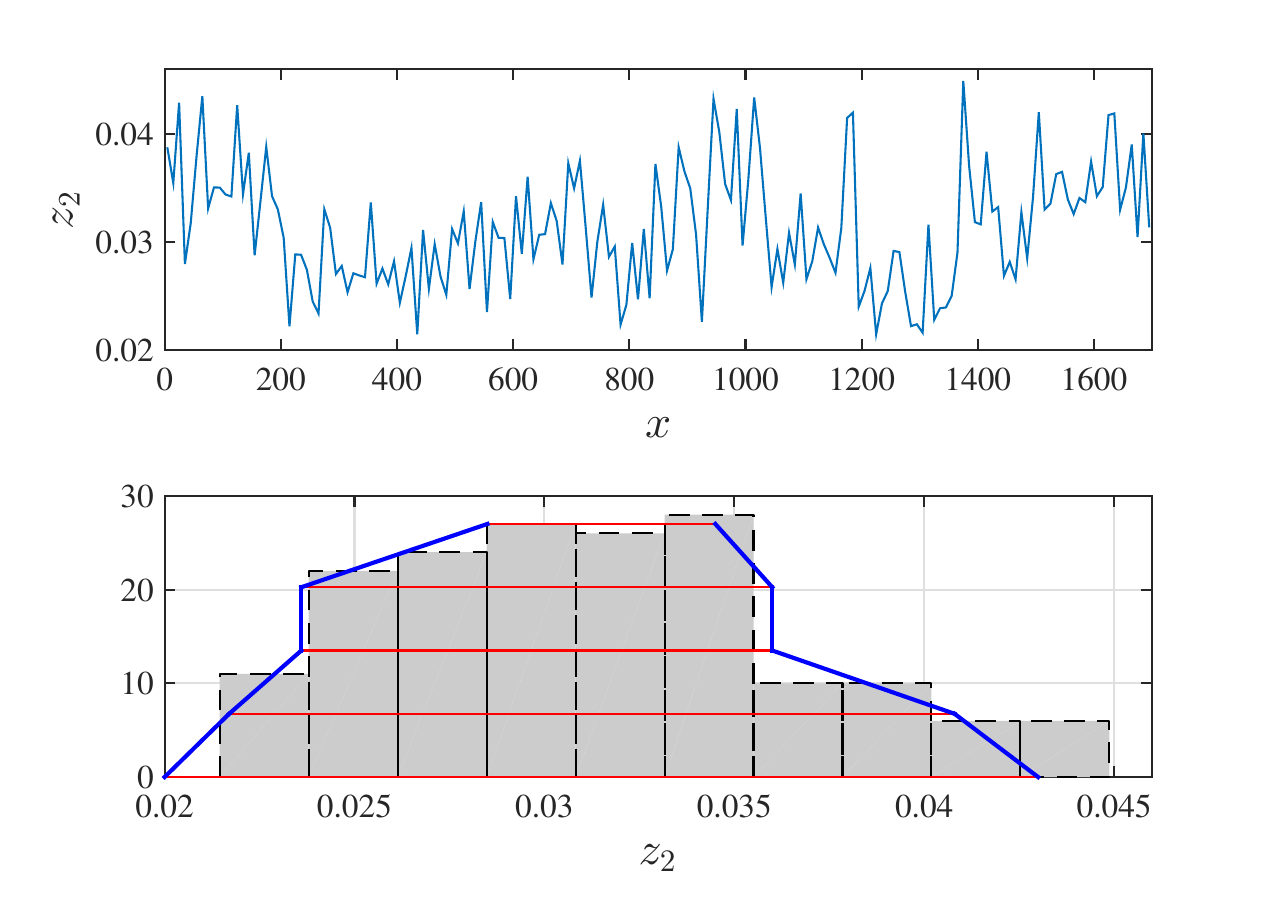}}\\
        \vspace{-.4cm}
        \subfigure[sample skewness]{\includegraphics[width=6.2cm,height=4.3cm]{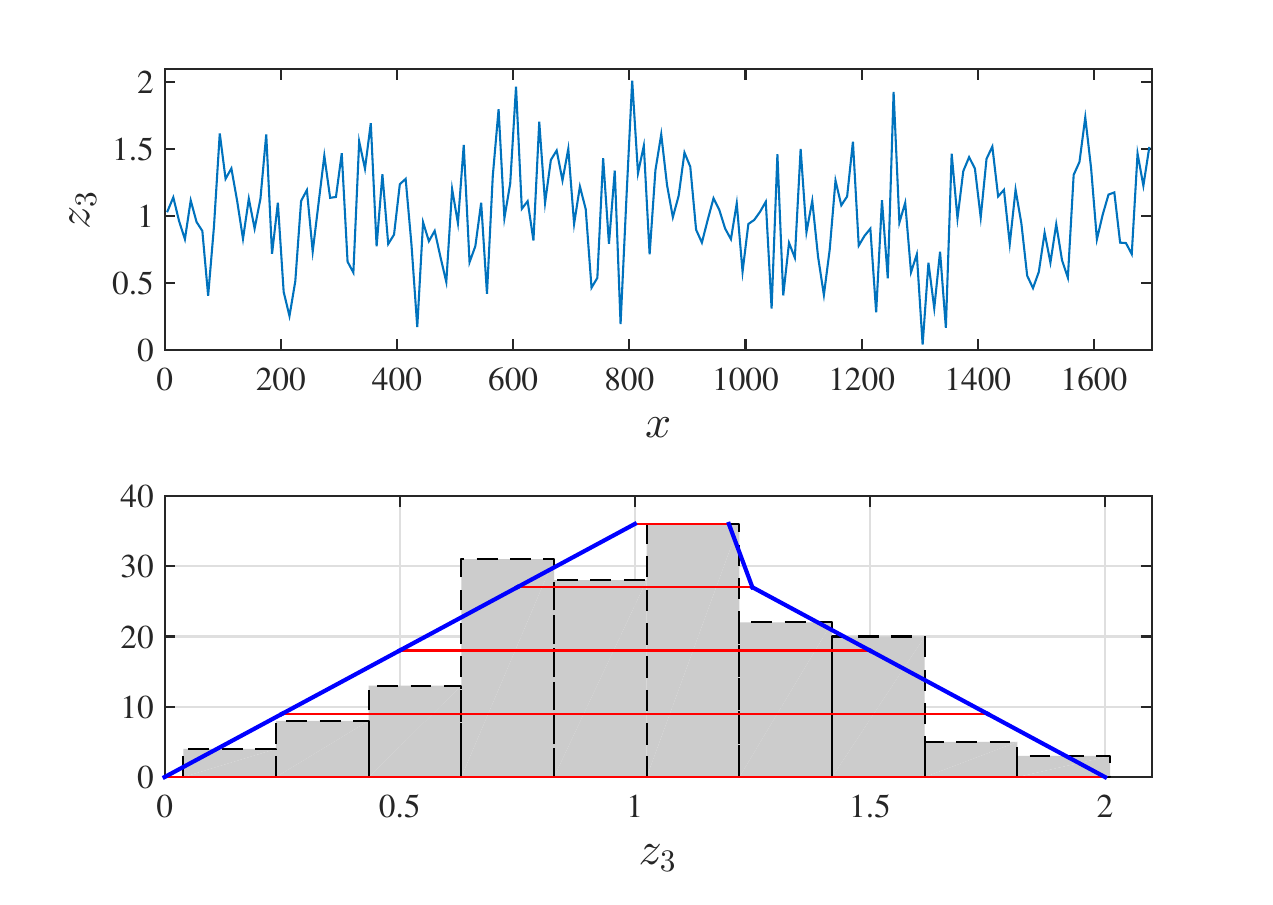}}
        \subfigure[sample excess kurtosis]{\includegraphics[width=6.2cm,height=4.3cm]{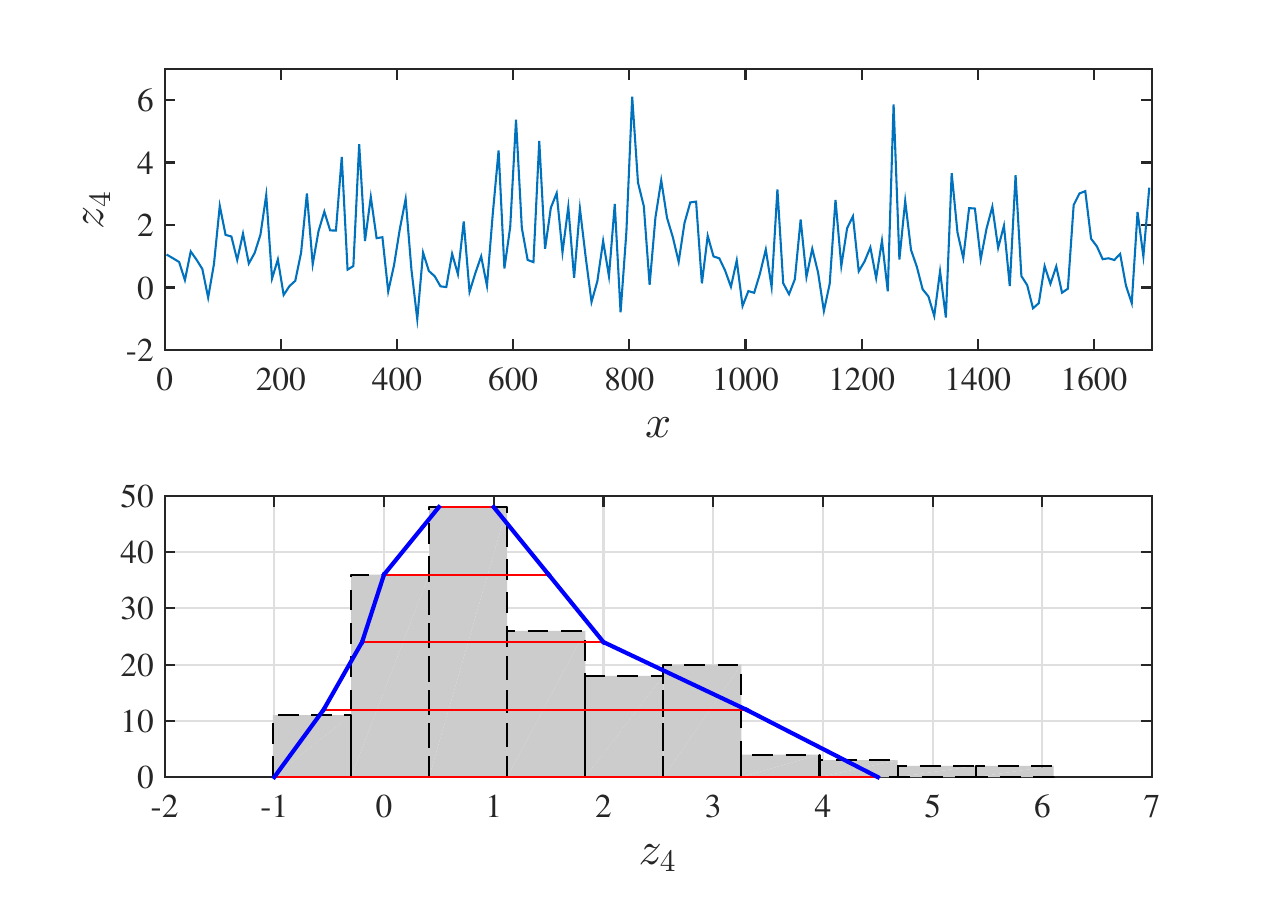}}
\vskip -0.35cm        
\caption{Sample moments of the field $b(x) = a^{-1}(x)$ versus $x$ and their histograms.}
\label{sample_moments}
  \end{center}
\end{figure}

A common engineering practice is to model material parameters, such as the compliance $b$, by stationary Gaussian random fields; see e.g. \cite{XU_Graham-Brady:05,Yin_etal:09,Liu_etal:10,Papadrakakis_Stefanou:14}. However, as Figure \ref{sample_moments} shows, the moments are not constant and vary in $x$, and hence the parameter $b$ cannot be accurately represented by stationary random fields. Moreover, the field is not Gaussian, since its skewness is not zero. One option within the framework of precise probability is to construct a non-stationary non-Gaussian random field. 
This option would heavily rely on the availability of abundant high-quality data to correctly capture the highly oscillatory moments. In reality such data are not available, for example when characterizing permeability of porous rock layers or compliance of composites containing millions of fibers. Even if-in the non-realistic absence of epistemic uncertainty-we do represent $b(x)$ by a non-stationary random field, we would face a stochastic multiscale problem that may not be tractable. This is due to the well-known fact that stochastic homogenization, necessary to treat random multiscale parameters, is applicable only to stationary fields \cite{Souganidis:99}. 
A second option is to construct an imprecise probabilistic model by considering a family of distributions, as done for example in interval probability \cite{Weichselberger:2000} and optimal UQ \cite{OUQ:2013}. Although these approaches can handle epistemic uncertainty, they may suffer from the loss of information and the non-propagation of uncertainty across multiple scales. Intuitively, this is because such models may not capture all input information that is available to us and hence cannot propagate the whole information. From the sample moments in Figure \ref{sample_moments} it is obvious that one would lose information if the moments are modeled by intervals. For instance if one models the first moment with the interval $[0.123, 0.156]$, then the information that the value $0.135$ is more possible/plausible would be lost. See also \cite{OUQ:2013} for an illustrative example of the non-propagation of uncertainty across multiple scales. Among the imprecise probabilistic models, second-order hierarchical models \cite{Gelman_etal:04} may be capable of treating this multiscale problem. In this case one may be able to model $b$ by a random field with random moments. 
Here, we propose another alternative beyond the framework of probability. In order to accurately model and propagate uncertainty and afford multiscale strategies, we propose to model the parameter $b$ by a {\it fuzzy-stationary random field} as follows. 

We first {\it fuzzify} the moments of $b$: we use the histograms of the sample moments to construct membership functions $\mu_{\tilde{z}_1}, \mu_{\tilde{z}_2}, \mu_{\tilde{z}_3}, \mu_{\tilde{z}_4}$. This can be done, for instance, by the method of least squares and with piecewise-linear regression functions (the thick blue lines in Figure \ref{sample_moments}). We then normalize the regression functions so that the maximum membership function value is one. It is to be noted that this procedure generates an initial draft for membership functions. We may need to conduct a subsequent modification and make additional corrections, for instance if the initial draft is not quasi-concave. Here, we use five $\alpha$-levels (0, 0.25, 0.5, 0.75, 1) for the construction and obtain four decagonal fuzzy variables, described by their ten vertices
\begin{align*}
\tilde{z}_1 &= \langle 
0.1222, \, 0.1249, \, 0.1277, \, 0.1304, \, 0.1330, \, 0.1360, \, 0.1388, \, 0.1445, \, 0.1502, \, 0.1559 
\rangle, \\
\tilde{z}_2 &= \langle 
0.0200, \, 0.0217, \, 0.0236, \, 0.0236, \, 0.0285, \, 0.0345, \, 0.0360, \, 0.0360, \, 0.0408, \, 0.0430
\rangle, \\
\tilde{z}_3 &= \langle 0, \, 0.25, \, 0.50, \, 0.75, \, 1.00, \, 1.20, \, 1.25, \, 1.50, \, 1.75, \, 2.00
\rangle, \\
\tilde{z}_4 &= \langle 
-1.00, \, -0.55, \, -0.20, \, 0, \, 0.50, \, 1.00, \, 1.50, \, 2.00, \, 3.30, \, 4.50
\rangle.
\end{align*}



We note that the four fuzzy variables are fully interactive, because the four moments are obtained from the same set of data $\{ b_{i}(x_j) \}_{i=1}^{50} $ and hence are directly related to each other, that is, higher moments are obtained from lower moments. This will result in a reduction in fuzzy space dimension. While we have a vector of four fuzzy variables $\tilde{\bf z}=(\tilde{z}_1,\tilde{z}_2,\tilde{z}_3,\tilde{z}_4 )$, their joint $\alpha$-cut $S_{\alpha}^{\tilde{\bf z}}$ is a piecewise linear one-dimensional curve embedded in ${\mathbb R}^4$. Similar to the numerical example 1 and using the arc length parameterization of the curve, we can represent $S_{\alpha}^{\tilde{\bf z}}$ by a piecewise linear map. 

We then construct a fuzzy-stochastic translation field to model the compliance:
\begin{equation}\label{b_FS}
b(x, {\bf y}, \tilde{\bf z}) = \Psi^{-1}({\tilde {\bf z}}) \circ \Phi (G(x,{\bf y})).
\end{equation}
Here, $\Psi (\tilde{\bf z})$ is the CDF of a four-parameter beta distribution determined by the four fuzzy moments, $\Phi$ is the standard normal CDF, and $G(x,{\bf y})$ is a standard Gaussian field, approximated  by the truncated KL expansion: $
G(x,{\bf y}) \approx \sum_{j=1}^{m} \sqrt{\lambda_j} \, \phi_j(x) \, y_j$, with $y_j \sim {\mathcal N}(0,1)$ and the eigenpairs $\{(\lambda_j, \phi_j(x) )\}_{j=1}^m$ of the deterministic covariance 
\begin{equation}\label{cov_function}
C(x_1,x_2) = \exp{\bigl(\frac{-|x_1 - x_2|^{p}}{2 \, \ell^2}\bigr)}, \qquad p=2, \quad \ell = 20 \, \mu m.
\end{equation}
We note that the selection of the covariance function and its parameters, such as the exponent $p$ and correlation length $\ell$, must be based on a systematic calibration-validation strategy; see \cite{BNT:08_Sandia,BM:16}. As we will see in Figure \ref{Ex2_1D_Q45}, the choice \eqref{cov_function} here delivers output quantities which fit the true quantities. Here, we choose $m=27$ KL terms to preserve $90\%$ of the unit variance of the Gaussian field $G$.

The construction \eqref{b_FS} has several advantages. First, it benefits from the simplicity of working with a stationary Gaussian field $G(x,{\bf y})$. Moreover, by applying the inverse of $\Psi$ on $\Phi (G) \in [0, 1]$, we obtain a field that achieves the target marginal fuzzy CDF $\Psi(\tilde{\bf z})$. Finally, since the fuzzy moments are $x$-independent, the field \eqref{b_FS} may be thought of as a {\it fuzzy-stationary} random field. One can hence employ global-local homogenization methods \cite{BM:16} and perform multiscale computations if needed.

We now let $L=1.7 \times 10^{-3}$m and consider the problem \eqref{FSPDE_1D} with the fuzzy-stochastic parameter $a=b^{-1}$ given by \eqref{b_FS}. Hence, the analytical solution \eqref{FSPDE_1D_sol} reads 
$u(x, {\bf y}, \tilde{\bf z}) = \int_0^x b(\xi, {\bf y} , \tilde{\bf z}) \, d\xi$.

We consider the following QoIs
\begin{align*}
& \tilde{Q}_4(x) = Q_4(x,\tilde{\bf z}) = {\mathbb E}[u(x,{\bf y},\tilde{\bf z})], \\
& \tilde{Q}_5(y) = Q_5(y,\tilde{\bf z}) = u(L/4,y,\tilde{\bf z}), \\
& \tilde{Q}_6 = Q_6(\tilde{\bf z}) = P(u(L/4,{\bf y},\tilde{\bf z}) \ge u_{\text{cr}}).
\end{align*}
Here, $\tilde{Q}_4$ is a fuzzy field, $\tilde{Q}_5$ is a fuzzy-stochastic function, and $\tilde{Q}_6$ is a fuzzy failure probability. We now discuss the computation of the above three quantities.

The computation of $\tilde{Q}_4$ and $\tilde{Q}_5$ is similar to that of $\tilde{Q}_2$ and $\tilde{Q}_3$ in Section \ref{sec:numerics1}. 
Figure \ref{Ex2_1D_Q45} shows the fuzzy field $\tilde{Q}_4(x)$ versus $x \in [0,1000] \mu$m (left) and the fuzzy CDF of $\tilde{Q}_5$ for three membership degrees $\alpha=0,0.5,1$. We also compute and plot the ``true'' quantities directly obtained by the real data, i.e. the 50 discrete samples, as follows. First, we choose $N_b=20$ groups of samples, where each group consists of $M_b = 15$ different, randomly selected samples out of 50 discrete samples. For each group we then compute $M_b$ samples of the true quantity and then obtain their expected value (to compare with $\tilde{Q}_4$) and their CDF (to compare with $\tilde{Q}_5$). This gives us a set of $N_b$ benchmark solutions, referred to as the ``truth". It is to be noted that the variations in true quantities reflect the presence of epistemic uncertainty and justify the need for models beyond precise probability. For instance for $Q_5$ we obtain a range of distributions, hence forming a nested set of p-boxes, instead of one single distribution that one may obtain in the absence of epistemic uncertainty. Figure \ref{Ex2_1D_Q45} shows how accurately the computed quantities obtained by the proposed fuzzy-stochastic PDE model capture the variations in the true quantities.
\begin{figure}[!h]
\center
\subfigure{\includegraphics[width=6.4cm,height=3.8cm]{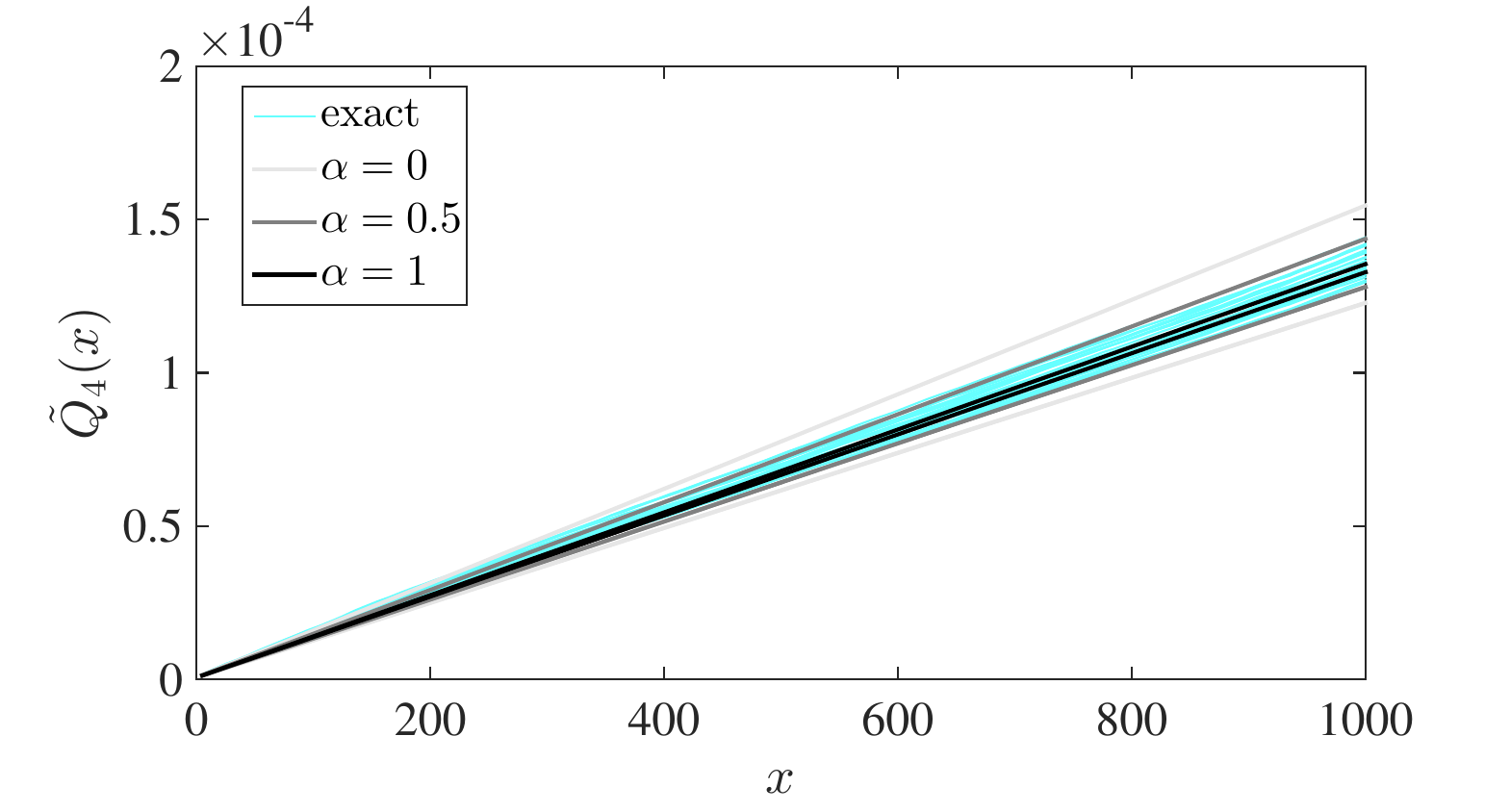}}
\hskip .1cm
\subfigure{\includegraphics[width=6.4cm,height=3.8cm]{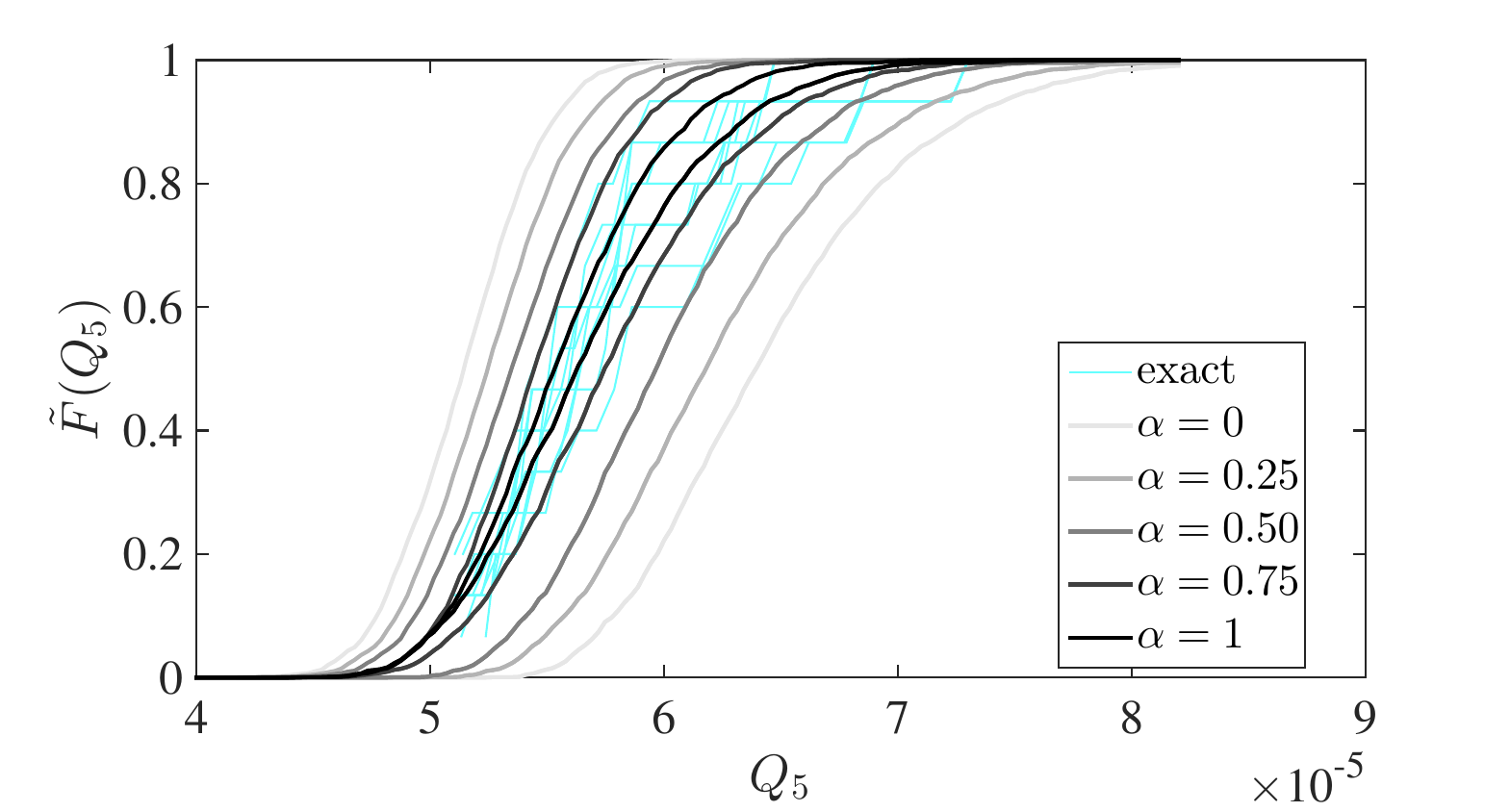}}
\vspace{-.6cm}
    \caption{The fuzzy field $\tilde{Q}_4(x)$ versus $x$ (left) and the fuzzy CDF of $\tilde{Q}_5$ (right). For comparison, the true quantities (thin turquoise curves) are included.}
    \label{Ex2_1D_Q45}
    \vskip -.4cm
 \end{figure}

The quantity $\tilde{Q}_6$ is the fuzzy probability of failure that would occur when the displacement $u$ at $x=L/4$ reaches a critical value $u_{\text{cr}}$. At every fixed $\alpha$-level, we first uniformly discretize the one-dimensional joint $\alpha$-cut $S_{\alpha}^{\tilde{\bf z}}$ into $M_f$ discrete points $\{ {\bf z}^{(k)} \}_{k=1}^{M_f} \in S_{\alpha}^{\tilde{\bf z}}$. 
We next set $g({\bf y},\tilde{\bf z}) := u_{\text{cr}} - u(L/4,{\bf y},\tilde{\bf z})$ and follow Algorithm \ref{ALG_fuzzy_random_functions} with $q({\bf y},\tilde{\bf z}) ={\mathbb I}_{[g({\bf y},\tilde{\bf z}) \le 0]}$. We use Monte Carlo sampling with $M_s$ realizations $\{ {\bf y}^{(i)} \}_{i=1}^{M_s}$ to approximate 
\vspace{-.5cm}
\begin{equation}\label{Q6_1}
Q_6({\bf z}^{(k)}) = {\mathbb E}[{\mathbb I}_{[g({\bf y},{\bf z}^{(k)}) \le 0]}] \approx \frac{1}{M_s} \sum_{i=1}^{M_s} {\mathbb I}_{[g({\bf y}^{(i)},{\bf z}^{(k)}) \le 0]}, \qquad k=1, \dotsc, M_f.
\vspace{-.2cm}
\end{equation}
The output $\alpha$-cut for $\tilde{Q}_6$ is then obtained by 
\begin{equation}\label{Q6_2}
S_{\alpha}^{\tilde{Q}_6} = 
 \Bigl[ \min_{k}  Q_6({\bf z}^{(k)}), \, \max_{k}  Q_6({\bf z}^{(k)}) \Bigr].
\end{equation}
Figure \ref{Ex2_1D_Q6} shows the membership function of $\tilde{Q}_6$ obtained from the $\alpha$-cuts given in \eqref{Q6_1}-\eqref{Q6_2} computed for five $\alpha$ levels $\alpha = 0, 0.25, 0.5, 0.75, 1$, with $u_{\text{cr}} = 6.9 \times 10^{-5} \, \mu$m, $M_f = 181$, and $M_s = 10^4$. 
%
\begin{figure}[!h]
\vspace{-.4cm}
\center
{\includegraphics[width=6.4cm,height=3.8cm]{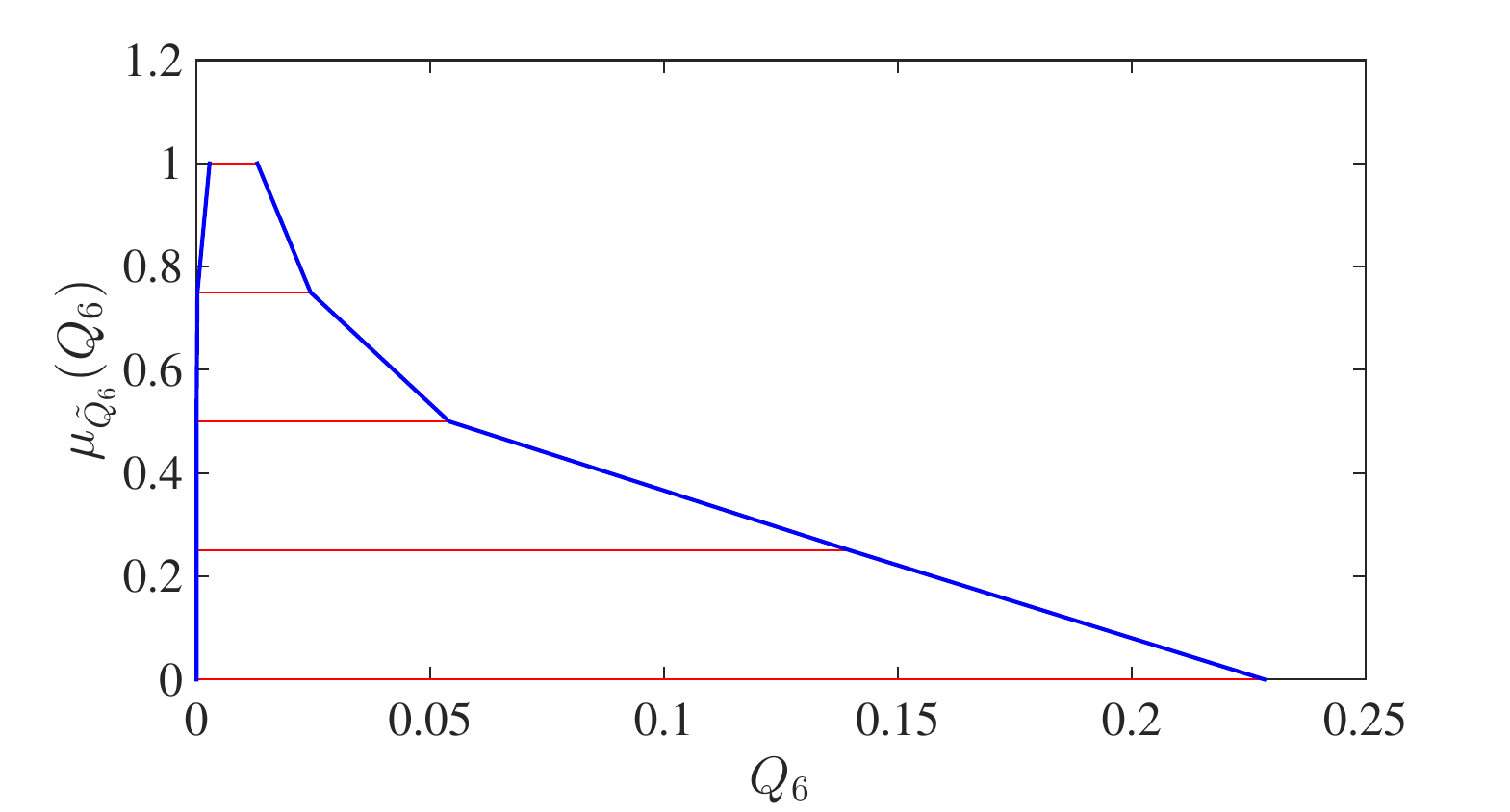}}
\vspace{-.2cm}
    \caption{The membership function of the fuzzy failure probability $\tilde{Q}_6$ and its five $\alpha$-cuts.}
    \label{Ex2_1D_Q6}
   \vskip -.3cm
 \end{figure}

Note that the additional nuanced information given through nested intervals at different levels of possibility of $\tilde{Q}_6$ is a direct result of the propagation of additional nuanced information available in the statistical moments in Figure \ref{sample_moments}. Such additional information may not be accounted for and hence would not be propagated though other imprecise probabilistic models, such as interval probabilities and optimal UQ. 
%
%
%
This is particularly important in ``certification problems", where we need to certify or decertify a system of interest. To illustrate this, let $\varepsilon_{\small{\text{TOL}}} = 0.1$ be the greatest acceptable failure probability $\tilde{Q}_6$, that is, the system is {\it safe} if $\tilde{Q}_6 \le \varepsilon_{\small{\text{TOL}}}$ and {\it unsafe} if $\tilde{Q}_6 > \varepsilon_{\small{\text{TOL}}}$. Suppose that the lower and upper bounds of the zero-cut of $\tilde{Q}_6$, i.e. $0$ and $0.2284$, represents the crisp lower and upper bounds obtained by an imprecise probabilistic approach. In this case since $0 < \varepsilon_{\small{\text{TOL}}} = 0.1 < 0.224$, then we cannot decide on the safety of the system, unless additional information will be provided. However, the additional nuanced information provided by the lower and upper bounds at different levels of possibility (i.e. different $\alpha$-levels) may help decision-makers. 
In fact the highest level of possibility (the most possible scenario) corresponding to the 1-cut in Figure \ref{Ex2_1D_Q6} suggests that the system may be safe. 


%
%
%

\subsection{Computational cost}
\label{sec:numerics_cost}


Consider a fuzzy-stochastic function $q({\bf y}, \tilde{\bf z})$, for example obtained by applying a combination of algebraic, integral, and differential operators on the solution $u({\bf x}, {\bf y}, \tilde{\bf z})$ to a fuzzy-stochastic PDE problem. Assume that we are interested in computing a fuzzy quantity $\tilde{Q} = {\mathbb E}[q({\bf y}, \tilde{\bf z})]$. The computation of one $\alpha$-cut $S_{\alpha}^{\tilde{Q}}$ requires $M_f$ function evaluations $Q({\bf z}^{(k)}) = {\mathbb E}[q({\bf y}, {\bf z}^{(k)})]$ at $M_f$ discrete points $\{ {\bf z}^{(k)} \}_{k=1}^{M_f} \in S_{\alpha}^{\tilde{\bf z}}$. At each discrete point the expectation of $q$ needs to be approximated by a sampling technique using $M_s$ samples. In total we need to solve $M = M_f \, M_s$ deterministic PDE problems. The size of $M$ depends mainly on the number of random variables $m$, the number of fuzzy variables $n$, and the regularity of $q$ with respect to ${\bf y}$ and $\bf z$. When $M$ is very large, e.g. in the absence of high regularity or when $m$ and $n$ are large, the computations may be prohibitively expensive. There are however practical situations where fuzzy-stochastic computations are feasible:
\begin{itemize}
\item[1.] In many applications we have a low-dimensional fuzzy space, i.e. $n \ll m$. A typical example is when we model an uncertain parameter, such as the compliance in Section \ref{sec:numerics2}, by a hybrid fuzzy-stochastic field. In this case, $n$ is usually 1 (if the moments are fully interactive) or 2-4 (if we use 2-4 non-interactive moments), while $m$ may be rather large depending on the correlation length of the field. As a result, fuzzy-stochastic computations are usually not much more expensive compared to solving purely stochastic problems. 

\item[2.] When $q$ is highly regular with respect to $({\bf y}, {\bf z})$ we can employ spectral methods on sparse grids instead of Monte Carlo sampling strategies in order to speed up the computations. The stochastic and fuzzy spaces are different considering we extract statistical information in the stochastic space and perform optimization in the fuzzy space. Therefore, thanks to high regularity, we may build a surrogate model of $q({\bf y},{\bf z})$ on the tensor product of two separate sparse grids (one on each space) using sparse interpolating polynomials; see e.g.  \cite{Motamed_etal:13,Motamed_etal:15}. 
This type of separation affords efficient extraction of statistical information and sparse optimization \cite{Interval:01,Incomplete_search:03}.

\item[3.] Current probabilistic models are usually not applicable to multiscale problems with highly oscillatory uncertain parameters. It is well known that stochastic homogenization, necessary to treat multiscale stochastic parameters, is applicable only to stationary random fields \cite{Souganidis:99}, while in many multiscale problems the moments are not constant. Another related problem with imprecise probabilistic models is the non-propagation of uncertainty across multiple scales \cite{OUQ:2013}. In such cases, a hybrid fuzzy-stochastic model may be considered as a feasible and accurate approach to treat the problem; see \cite{BM:16} where a fuzzy-stochastic multiscale approach is presented for fiber composite polymers.

\item[4.] Due to the non-intrusiveness of the numerical methods, the $M$ deterministic problems can be distributed and solved independently on parallel computers. 
\end{itemize}

\section{Conclusion}
\label{sec:CON}

We have introduced a new class of PDE problems with hybrid fuzzy-stochastic parameters, coined fuzzy-stochastic PDE problems. Using the level-set representation of fuzzy functions, we have defined the solution to fuzzy-stochastic PDE problems through a corresponding parametric problem and further studied the analysis and computation of such problems.

The author is currently working on the development of more efficient numerical methods for solving fuzzy-stochastic PDE problems in line with the discussion presented in Section \ref{sec:numerics_cost}. More applications involving time-dependent fuzzy-stochastic PDEs will also be explored and presented elsewhere. 
While the focus of the present work has been mainly on forward problems, the construction of a fuzzy Bayesian computation algorithm 
for the inverse propagation of hybrid uncertainties through fuzzy-stochastic PDE problems is a subject of current work and will be presented elsewhere.

\section*{Acknowledgments}
I am indebted to Ivo Babu{\v s}ka for introducing me to the field of fuzzy calculus and for our regular stimulating discussions on uncertainty quantification and modeling. I would also like to acknowledge the invaluable comments and suggestions of the SIAM/ASA JUQ editors and reviewers that have greatly improved the manuscript.

\bibliographystyle{plain}
\bibliography{refs}

\end{document}